\documentclass[12pt]{amsart}
\usepackage{setspace}
\usepackage{amssymb,color}
\usepackage{graphicx,hyperref}
\usepackage{enumerate}
\usepackage{multirow}
\usepackage[
 margin=2 cm,
 includefoot,
 footskip=30pt,
]{geometry}

\newtheorem{thm}{Theorem}[section]
\newtheorem{prop}[thm]{Proposition}
\newtheorem{lem}[thm]{Lemma}

\newtheorem{defn}[thm]{Definition}

\theoremstyle{definition}

\newtheorem{example}[thm]{Example}

\theoremstyle{remark}

\newcommand\less{<_{\mathrm{alt}}}
\newcommand\lesseq{\le_{\mathrm{alt}}}
\newcommand\gess{>_{\mathrm{alt}}}
\newcommand\gesseq{\ge_{\mathrm{alt}}}
\renewcommand{\S}{\mathcal{S}}
\newcommand{\C}{\mathcal{C}}
\DeclareMathOperator{\Al}{Allow}
\newcommand{\B}{\overline{B}}
\newcommand{\N}{\overline{N}}
\renewcommand{\b}{\bar{b}}
\renewcommand{\P}{\overline{P}}
\newcommand{\Tb}{T_{-\beta}}
\newcommand{\Sb}{\Sigma_{-\beta}}
\newcommand{\Wb}{\mathcal{W}_{-\beta}} 
\newcommand{\WN}{\mathcal{W}_{N}}
\DeclareMathOperator{\asc}{asc}
\DeclareMathOperator{\Pat}{Pat}
\newcommand\wmax{\Omega_\beta}
\newcommand\wmin{\omega_\beta}

\newcommand \wend{\nu_{\beta}} 

\newcommand\wmaxn{((N{-}1)0)^\infty}    
\newcommand\wminn{(0(N{-}1))^\infty}   
\newcommand \betaone{ \mathfrak{a}_{\beta}}

\title{Patterns of negative shifts and beta-shifts}
\author{Sergi Elizalde}%
\email{sergi.elizade@dartmouth.edu}%
\author{Katherine Moore}%
\email{katherine.e.moore.gr@dartmouth.edu}%
\address{Department of Mathematics, Dartmouth College, Hanover, NH 03755, USA}%

\begin{document}

\maketitle

\begin{abstract}

The $\beta$-shift is the transformation from the unit interval to itself that maps $x$ to the fractional part of $\beta x$.
Permutations realized by the relative order of the elements in the orbits of these maps have been studied in~\cite{Elishifts} for positive
integer values of $\beta$ and in~\cite{Elibeta} for real values $\beta>1$. In both cases, a combinatorial description of the smallest positive value of $\beta$ needed to realize a permutation is provided. 
In this paper we extend these results to the case of negative $\beta$, both in the integer and in the real case. Negative $\beta$-shifts are related to digital expansions with negative real bases, studied by Ito and Sadahiro~\cite{ItoS}, and Liao and Steiner~\cite{LS}. 
\end{abstract}

\section{Introduction}

The study of the permutations realized by the one-dimensional dynamical systems provides a important tool to distinguish random from deterministic time series, as well as a combinatorial method to compute the topological entropy of the dynamical system. 

If $X$ is a linearly ordered set, $f:X\to X$ a map, and $x\in X$, we can consider the finite sequence $x,f(x),f(f(x)),\dots,f^{n-1}(x)$. If these $n$ values are different, then their relative order determines a permutation $\pi\in\S_n$, obtained by replacing the smallest value by a 1, the second smallest by a 2, and so on. 
We write $\Pat(x,f,n)=\pi$, and we say that $\pi$ is an {\em allowed pattern} of $f$, or that $\pi$ is {\em realized} by $f$, and also that $x$ induces $\pi$. If there are repeated values in the first $n$ iterations of $f$ starting with $x$, then $\Pat(x,f,n)$ is not defined. The set of allowed patterns of $f$ is
$$\Al(f)=\bigcup_{n\ge0}\{\Pat(x,f,n):x\in X\}.$$

It was shown in~\cite{Bandt} that if $X$ is an interval of the real line and $f$ is a piecewise monotone map, then there are some permutations that are not realized by $f$, called the forbidden patterns of $f$. Additionally, the growth rate of the sequence that counts allowed patterns by length gives the topological entropy of $f$, which is a measure of the complexity of the associated dynamical system.

Determining the set of allowed patterns for particular families of maps is a difficult problem in general, and an active area of research. In recent years it has been solved for shift maps~\cite{AEK,Elishifts} and for $\beta$-shifts~\cite{Elibeta}, and there has been some progress for signed shifts~\cite{Amigosigned,ArcEli,KArc} and logistic maps~\cite{Eliu}.

Shift maps can be described as maps of the form $f:[0,1]\to[0,1]$, $f(x)=\{Nx\}$, where $N$ is a positive integer and $\{y\}=y-\lfloor y \rfloor$ denotes the fractional part of $y$. They can also be interpreted as shifts of infinite words on an $N$-letter alphabet, where the linear order on the set is the lexicographic order. In~\cite{Elishifts}, a simple formula is given to determine, for a given permutation $\pi$, the smallest positive integer $N$ such that $\pi$ is realized by the shift on $N$ letters. This formula is then used to count the number of permutation of a given length realized by such a shift.

A natural generalization of shifts are $\beta$-shifts, which are the maps obtained when we replace $N$ by a an arbitrary real number $\beta>1$. They have their origin in the study of expansions of real numbers in an arbitrary real base $\beta>1$, introduced by R\'enyi~\cite{Ren} (see also~\cite{Par}). 
In~\cite{Elibeta}, a method is given to compute, for a given permutation $\pi$, the smallest positive real number $B(\pi)$ such that $\pi$ is realized by the $\beta$-shift for all $\beta>B(\pi)$. This number is called the {\em shift-complexity} of $\pi$ in~\cite{Elibeta}.

Signed shifts are a different generalization of shift maps, where some of the slopes in the graph of $f$ are allowed to be negative. The tent map is a particular case of a signed shift, but no formula is known for the number of its allowed patterns of a given length.
The only case of a signed shift (other than the one with all positive slopes) for which the number of allowed patterns is known 
is when all the slopes are negative. With the above definition of fractional part, these negative shifts can be defined as $f(x)=\{-Nx\}$ for an integer $N\ge2$. The enumeration of allowed patterns is solved in~\cite{KArc} for $N=2$ and in~\cite{EMupcoming} for the general case.

In this paper we focus on a variation of $\beta$-shifts, called {\em negative $\beta$-shifts}. For $\beta\in\mathbb{R}_{>1}$, the $-\beta$-transformation is defined as
\begin{equation}\label{eq:Tb}
\Tb : (0, 1] \rightarrow (0, 1],\quad x \mapsto -\beta x+\lfloor\beta x \rfloor+1=1-\{\beta x\}.
\end{equation}
The graph of $T_{-(1+\sqrt{2})}$ is shown in Figure~\ref{fig:Tb}. The map $\Tb$ map agrees in all but a finite set of points with the transformation $x\mapsto \{-\beta x\}$ from $[0,1)$ to itself, which has been studied in~\cite{Gora}.
\begin{figure}[h]
\centering
\includegraphics[totalheight=55mm]{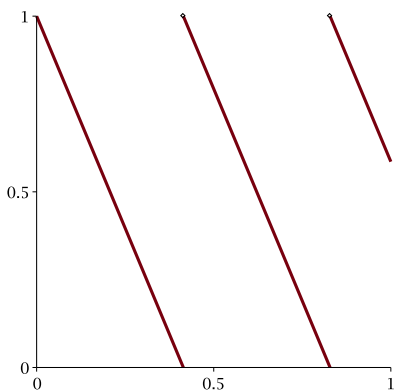}
\caption{The graph of $\Tb$ for $\beta=1+\sqrt{2}$.}
\label{fig:Tb}
\end{figure}

We will see that, as we increase $\beta$, the set of allowed patterns of $\Tb$ grows (in the sense of containment), analogously to the situation for the regular $\beta$-shift.
Given a permutation $\pi$, our goal is to find the smallest value  $\B(\pi)$ such that $\pi\in\Al(T_{-\beta})$ for all $\beta>\B(\pi)$. Our approach is similar to the one used in~\cite{Elibeta} for the positive $\beta$-shift, but there are some intricacies that appear only in the negative case. In particular, one has to consider several different cases depending on the shape of $\pi$.

Negative $\beta$-shifts are closely related to digital expansions with negative real bases, which were introduced by Ito and Sadahiro~\cite{ItoS}. Liao and Steiner studied dynamical properties of the transformation $\Tb$ in~\cite{LS}. More recently, Steiner~\cite{Steiner} characterized the sequences that occur as the digital expansions of $1$  with base $-\beta$ for some $\beta>1$, which is important when determining what sequences are admissible as $-\beta$-expansions (in analogy with Parry's work for the positive case~\cite{Par}).

An important special case of negative $\beta$-shifts, which is also a particular case of signed shifts, occurs when $\beta$ is an integer, $\beta=N\ge2$. In Section~\ref{sec:-N} we study this map, and we determine, for a given permutation $\pi$,  the smallest value of $N\ge2$ such that $\pi$ is realized by the corresponding negative shift.
In Section~\ref{sec:beta} we move to the case of real $\beta$, and we consider the sequences that can be obtained as representations of real numbers in base $-\beta$, in order to interpret negative $\beta$-shifts as shifts on infinite words in a certain set $\Wb$. In Section~\ref{sec:words} we give a construction that, for a given permutation $\pi$, provides a word in $\Wb$ that induces $\pi$  and represents a number in base $-\beta$ for the smallest possible $\beta$. Finally, in Section~\ref{sec:B} we provide a formula for the number $\B(\pi)$ described above as the largest root of a certain polynomial.

In the rest of the paper, $\pi$ denotes a permutation in the symmetric group $\S_n$.

We remark that after an earlier version of this paper was posted on \texttt{arxiv.org}, we were informed by Charlier and Steiner that they have independently obtained similar results~\cite{CSpreprint}.

\section{The reverse shift}\label{sec:-N}

When $\beta$ is an integer, which we denote by $\beta=N\ge2$, we give a slightly different definition of the negative shift. Let 
$$ M_{-N} : [0, 1] \rightarrow [0, 1],\quad x \mapsto \left\{
\begin{array}{ll}
     1-\{Nx\} & \text{ if } x \in [0, 1), \\
         0 & \text{ if } x = 1, \\
\end{array} 
\right. $$
and call this map the {\em reverse shift}. 
Note that $M_{-N}(x)=T_{-N}(x)$ for all $x\in(0,1)$, and so $\Al(M_{-N})=\Al(T_{-N})$. We choose to use the map $M_{-N}$ for consistency with the definition of signed shifts used in \cite{Amigosigned,ArcEli}, and also to avoid the isolated point $T_{-N}(1)=1$.

For an integer $N\ge2$, let $\WN$ be the set of infinite words on the alphabet $\{0,1,\dots,N{-}1\}$, equipped with
the {\em alternating lexicographic order}, which is defined by $v_1v_2\dots \less w_1w_2\dots$ if there exists some $i$ such that $v_j = w_j$ for all $j < i$ and $(-1)^i(v_i -w_i)>0$.
Let $\Sigma_{-N}$ be the shift map on $(\WN,\less)$, defined as $\Sigma_{-N}(w_1 w_2 w_3 \dots ) = w_2 w_3 \dots$ for $w\in\WN$. 

Throughout this paper, we write $w=w_1w_2\dots$ and use the notation $w_{[k,h]}=w_kw_{k+1}\dots w_{h}$ and $w_{[k, \infty)} = w_k w_{k+1} \dots$.  If $d$ is a finite word, then $d^m$ denotes concatenation of $d$ with itself $m$ times, and $d^{\infty}$ denotes the corresponding infinite periodic word.  We say that a finite word $d$ is primitive if it cannot be written as a power of any proper subword, i.e., it is not of the form $d = a^m$ for any $m > 1$.  Equivalently, it is well known that a word $d$ is primitive if it is not equal to any of its non-trivial cyclic shifts.

Let
\begin{equation}\label{eq:WN0}
\WN^0 =  \WN \setminus \{w: w= w_1 w_2 \dots w_{k} (0(N{-}1))^\infty  \text{ and } w_k \neq N{-}1, \text{ for some }k\ge 1\},
\end{equation}
which is closed under shifts. The map $\Sigma_{-N}$ restricted to $(\WN^0,\less)$ is order-isomorphic to the map $M_{-N}$ on $([0,1],<)$, via the order-isomorphism $\psi: \WN^0\mapsto [0, 1]$ defined by $\psi(w_1 w_2 \dots ) = - \sum_{j = 1}^\infty \frac{w_j + 1}{(-N)^j}$.    

To see that $\psi$ is an order-isomorphism, let us first show that $M_{-N} \circ \psi = \psi \circ \Sigma_{-N}$ on $\WN^0$.
If $\psi(w_1 w_2 w_3 \dots ) \neq 1$,
\begin{align*}
M_{-N} \circ \psi(w_1 w_2 w_3 \dots) & =  M_{-N}\left( - \sum_{j = 1}^\infty \frac{w_j + 1}{(-N)^j} \right) 
=  1 - \left\{ N \left(-\sum_{j = 1}^\infty \frac{w_j + 1}{(-N)^j}\right) \right\} \\
& =  1 - \left\{  w_1 + 1 + \sum_{j = 1}^\infty \frac{w_{j+1} + 1}{(-N)^j} \right\}
= 1 - \left( 1 + \sum_{j = 1}^\infty \frac{w_{j+1} + 1}{(-N)^j} \right) \\
& =  -\sum_{j =1}^\infty \frac{w_{j+1} + 1}{(-N)^j}   
= \psi(w_2 w_3 w_4\dots)
= \psi \circ \Sigma_{-N}(w_1 w_2 w_3 \dots).
\end{align*}
If $\psi(w_1 w_2 w_3 \dots) = 1$, then $w_1 w_2 w_3 \dots = ((N{-}1) 0)^\infty$, and in this case
$$M_{-N} \circ \psi(w_1 w_2 w_3 \dots)  = M_{-N}(1) = 0$$
and 
$$\psi \circ \Sigma_{-N} (w_1 w_2 w_3 \dots ) = \psi(w_2 w_3 w_4 \dots) = \psi((0 (N{-}1))^{\infty}) = 0.$$

Next we show that $\psi$ is order-preserving.

\begin{lem} Let $v,w\in \WN^0$. If $v \less w$ , then $\psi(v) < \psi(w).$ \end{lem}

\begin{proof} Let $i$ be the index such that $v_j = w_j$ for all $j < i$ and $(-1)^i(v_i - w_i) > 0$.  Then 
\begin{align*}
\psi(w) - \psi(v) &= - \sum_{j = 1}^\infty \frac{w_j + 1}{(-N)^j} + \sum_{j = 1}^\infty \frac{v_j + 1}{(-N)^j} 
= - \frac{(w_i - v_i)}{(-N)^i} - \frac{1}{(-N)^i} \left( \psi(w_{[i+1,\infty)}) - \psi(v_{[i+1,\infty)} ) \right) \\
&= \frac{1}{N^i} \left( (-1)^i(v_i - w_i) + (-1)^i(\psi(v_{[i+1,\infty)}) - \psi(w_{[i+1,\infty)}) )\right) \geq  0,
\end{align*}
where the last inequality follows from the fact that $(-1)^i(v_i - w_i) \geq 1$ and \\ $|\psi(v_{[i+1,\infty)}) - \psi(w_{[i+1,\infty)}) | \leq 1$.  Moreover, if $i$ is even we have equality if and only if $v_{[i,\infty)} = v_i ((N{-}1) 0)^\infty$ and $w_{[i,\infty)} = (v_i  - 1) (0(N{-}1))^\infty$, in which case $w \notin \WN^0$.  If $i$ is odd, we have equality if and only if $w_{[i,\infty)} = w_i ((N{-}1)0)^\infty$ and $v_{[i,\infty)} = (w_i-1) (0(N{-}1))^\infty$, in which case $v \notin \WN^0$.  Therefore, the inequality is always strict.  
\end{proof} 

Even though we defined $\Sigma_{-N}$ on the larger set $\WN$, we will now show that the words $w\in\WN\setminus\WN^0$ do not induce any additional patterns, and so the above order-isomorphism implies that $\Al(M_{-N})=\Al(\Sigma_{-N})$. Recall that such words can be written as $w=w_1\dots w_k(0 (N{-}1))^\infty$ with $w_k \neq N{-}1$ and $k\ge 1$. If $k<n-2$, then $w$ does not induce any pattern of length $n$, because $w_{[n, \infty)} = w_{[n-2, \infty)}$. If $k \geq n-2$, then the word $w' = w_1 w_2 \dots w_{k} (0(N{-}1))^n 0^\infty \in \WN^0$ satisfies $\Pat(w', \Sigma_{-N}, n) = \Pat(w, \Sigma_{-N}, n)$.
For $k\ge n$, one could alternatively take $v=w_1\dots w_{k-1}(w_k + 1) ((N{-}1)0)^\infty$, which satisfies
$\Pat(v, \Sigma_{-N}, n) = \Pat(w, \Sigma_{-N}, n)$, and also $\psi(v) = \psi(w)$, when extending the above definition of $\psi$ to $\WN$.

The following lemma describes a straightforward property of the sets of allowed patterns of negative shifts.

\begin{lem}\label{lem:monotoneN} $\Al( \Sigma_{-N}) \subseteq \Al(\Sigma_{-(N+1)})$. \end{lem}

 \begin{proof} Let $\pi \in \Al(\Sigma_{-N})$.  Then there exists a word $w \in \WN\subseteq \mathcal{W}_{N+1}$ such that \\ $\Pat(w, \Sigma_{-N}, n) = \pi$. Since $\Sigma_{-(N+1)}$ and $\Sigma_{-N}$ are shift maps, they agree on the alphabet $\WN$.  Therefore,  $\Pat(w, \Sigma_{-(N+1)}, n) = \Pat(w, \Sigma_{-N}, n) = \pi$, and so $\pi \in \Al(\Sigma_{-(N+1)})$. \end{proof}

For a given permutation $\pi$, let 
$$ \N(\pi) = \min\{N : \pi \in \Al( \Sigma_{-N} ) \},$$
that is, the smallest positive integer $N$ such that $\pi$ is realized by $\Sigma_{-N}$. Our goal in this section is to give a formula for $\N(\pi)$.

For this purpose, we will use a bijection that was introduced in~\cite{Elishifts}.
Let $\C_{n}^{\star}$ be the set of cyclic permutations of $[n]$ with a distinguished entry. We use the symbol ${\star}$ to denote the distinguished entry, since its value can be recovered from the other entries, and we will use both one-line notation and cycle notation. For example,
the cycle $(2,1,3)=312$, with the entry 2 marked, becomes $({\star},1,3)=31{\star}\in\C_n^{\star}$.
Define a bijection $\S_n \rightarrow \C_n^{\star}$ by $\pi \mapsto \hat{\pi}$ where, if $\pi = \pi_1 \pi_2 \dots \pi_n$ in one-line notation, then $\hat{\pi} = ({\star}, \pi_2, \dots, \pi_n)$ in cycle notation. Note that $\hat\pi$ satisfies $\hat{\pi}_{\pi_i} = \pi_{i +1}$ for $1 \leq i \le n-1$, and $\hat{\pi}_{\pi_n}=\pi_1$, which is the entry marked with a ${\star}$. This section builds on the techniques used by Archer in~\cite{KArc}.

For $1\le j\le n-1$, we say that $j$ is an {\em ascent} of $\hat\pi$ if either $\hat{\pi}_{j} < \hat{\pi}_{j + 1}$, or $\hat{\pi}_{j + 1}  = {\star}$ and $\hat{\pi}_{j}  < \hat{\pi}_{j +2}$. In the latter case (which requires $j\le n-2$), we say that $j$ is an {\em ascent over the $\star$}. Denote by $\asc(\hat\pi)$ the number of ascents of $\hat\pi$. Similarly to how we define ascents of $\hat\pi$ by skipping the $\star$, we say that a sequence $\hat\pi_i\hat\pi_{i+1}\dots\hat\pi_j$ is decreasing if so is the sequence obtained after deleting the $\star$, if applicable.

\begin{defn} \label{segmentation} A $-N$-segmentation of $\hat{\pi}$  is a set of indices $0 = e_0 \leq e_1 \leq \dots \leq e_N = n$ such that
\begin{enumerate}[(a)]
\item the sequence $\hat{\pi}_{e_k +1}\hat{\pi}_{e_k +2} \dots \hat{\pi}_{e_{k+1}}$ is decreasing for all $0\le k<N$;
\item if $\hat{\pi}_1 = n$ and $\hat{\pi}_{n-1}\hat{\pi}_{n} = 1{\star}$, then either $e_1 =0$ or $e_{N{-}1} \geq  n-1$;
\item if $\hat{\pi}_n = 1$ and $\hat{\pi}_1 \hat{\pi}_2 = {\star}n$, then either $e_{N{-}1} = n$ or $e_1 \leq 1$.
\end{enumerate} 
To each $-N$-segmentation of $\hat{\pi}$ we associate a finite word $\zeta = z_1 z_2 \dots z_{n-1}$ defined by $z_{i} = k$ whenever $e_k < \pi_i \leq e_{k+1}$, for $1\le i\le n-1$.
\end{defn}

Notice that condition (a) forces a $-N$-segmentation to have an index for each ascent of $\hat{\pi}$. More precisely, if $j$ is an ascent of $\hat\pi$, then $e_i=j$ for some $i$, unless $j$ is an ascent over the $\star$, in which case $e_i\in\{j,j+1\}$ for some $i$.
It follows that in order for $\hat\pi$ to have an $-N$-segmentation, we must have $N\ge 1+\asc(\hat\pi)$.  

If conditions (b) and (c) do not hold, a $-N$-segmentation with $N = \asc(\hat{\pi}) + 1$ is called a {\em minimal segmentation} of $\hat\pi$. The minimal segmentation of $\hat\pi$ is unique unless $\hat\pi$ has an ascent $j$ over the $\star$, in which case there are two minimal segmentations, corresponding to the choice $e_i\in\{j,j+1\}$  described above. In this case we have $\hat\pi_{\pi_n}=\star=\hat\pi_{j+1}$, which implies that $\pi_n=j+1$, and so both minimal segmentations produce the same prefix $\zeta$. Thus, the prefix $\zeta$ produced by minimal segmentations is unique.

When we do not need to specify $N$, a $-N$-segmentation will simply be called a segmentation.

\begin{example}  
Let $\pi = 1572364$.  Then $\hat{\pi} = 536{\star}742$, whose ascents are $2$ and $3$, the latter being an ascent over the $\star$.  Therefore, $\hat{\pi}$ has two $-3$-segmentations (i.e., minimal segmentations) given by  $(e_0, e_1, e_2, e_3) = (0, 2, 3, 7)$, and by $(e_0, e_1, e_2, e_3) = (0, 2, 4, 7)$,
respectively. Both produce the prefix $\zeta = 022012$.  
\end{example}

We will show that, under certain circumstances, it is possible to complete the prefix $\zeta$ into a word in $w= \zeta w_{[n, \infty)}\in\WN$ such that $\Pat(w, \Sigma_{-N}, n) = \pi$.  

Given a $-N$-segmentation of $\hat{\pi}$ and its associated finite word $\zeta = z_{[1, n-1]}$, we define the following indices and subwords of $\zeta$.  If $\pi_n\neq n$, let $x$ be the index such that $\pi_x = \pi_n + 1$, and let $p = z_{[x,n-1]}$.  Similarly, if $\pi_n\neq 1$, let $y$ be such that $\pi_y = \pi_n - 1$, and let $q = z_{[y,n-1]}$. 

\begin{defn} A segmentation of $\hat{\pi}$ is {\em invalid} if the associated prefix $\zeta$ satisfies that both $p$ and $q$ are defined and either $p = q^2$ or $q = p^2$.  Otherwise the segmentation is {\em valid}. \end{defn}

Note that if one minimal segmentation is invalid, then so is the other (if there is more than one), since it produces the same prefix $\zeta$.
It will be convenient to classify permutations into three types as follows. 

\begin{defn} We say that $\pi$ is 
\begin{itemize}
\item {\em cornered} if either $\hat{\pi}_1 = n$ and $\hat{\pi}_{n-1}\hat{\pi}_n = 1 {\star}$, or $\hat{\pi}_n = 1$ and $\hat{\pi}_1\hat{\pi}_2 = {\star} n$ (equivalently, if either $\pi_{n-2} \pi_{n-1} \pi_n = (n{-}1) 1 n$ or $\pi_{n-2} \pi_{n-1} \pi_{n} = 2n1$, respectively);
\item {\em collapsed} if the minimal segmentations of $\hat{\pi}$ are invalid; 
\item {\em regular} if $\pi$ is neither cornered nor collapsed.  
\end{itemize} 
\end{defn}

Note that the conditions on $\hat{\pi}$ for $\pi$ to be cornered are the same as in cases (b) and (c) in Definition~\ref{segmentation}. We point out that a permutation cannot be simultaneously cornered and collapsed. Indeed, a collapsed permutation requires the words $p$ and $q$ to be defined, which only happens if $\pi_n \notin\{1, n\}$. On the other hand, cornered permutations require $\pi_n = 1$ or $\pi_n = n$.  In particular, a minimal segmentation of $\hat{\pi}$ is defined for both collapsed and regular permutations.  
We can now state the main result of this section.

\begin{thm} \label{negativeshift} We have
$$\N(\pi) = 1 + \asc(\hat{\pi}) + \epsilon(\hat{\pi})$$
where 
$$\epsilon(\hat{\pi}) = \begin{cases} 0 & \text{if $\pi$ is regular,}\\
1 & \text{if $\pi$ is cornered or collapsed.}\end{cases}$$
\end{thm}

The rest of this section is dedicated to proving Theorem~\ref{negativeshift}.  Lemmas \ref{onea} through \ref{geq} are used to prove that $\N(\pi) \geq 1 + \asc(\hat{\pi}) + \epsilon(\hat{\pi})$.  Lemma \ref{numberseg} also gives information about the number of distinct prefixes $\zeta$ associated to valid $-N$-segmentations of $\hat{\pi}$ when $N=1 + \asc(\hat{\pi}) + \epsilon(\hat{\pi})$, which will be important in Section~\ref{sec:words} when we calculate $\B(\pi)$, the analog of $\N(\pi)$ for the map $\Tb$.  In the remaining lemmas, we show that 
certain words $s,t,\in\mathcal{W}_{1 + \asc(\hat{\pi}) + \epsilon(\hat{\pi})}$ induce the pattern $\pi$.  This will allow us to conclude that $\N(\pi) = 1 + \asc(\hat{\pi}) + \epsilon(\hat{\pi})$.  

\begin{example}
Let $\pi = 345261$. Then $\hat{\pi} = {\star}64521$ and $\pi$ is cornered, so $\epsilon(\hat{\pi}) = 1$. Since $\asc(\hat\pi)=1$, Theorem~\ref{negativeshift} says that $\N(\pi)=3$. A $-3$-segmentation of $\hat\pi$ is given by $(e_0, e_1, e_2, e_3) = (0, 3, 6, 6)$, producing $\zeta = 01101$.  A different $-3$-segmentation is given by $(e_0, e_1, e_2, e_3) = (0, 0, 3, 6)$, producing $\zeta = 12212$. 
\end{example}

\begin{example} Let $\pi = 3651742$. Then $\hat{\pi} = 7{\star}62154$ and $\asc({\hat{\pi}})=1$. The only minimal segmentation, given by $(e_0, e_1, e_2) = (0, 5, 7)$,  produces the word $\zeta = 010010$, which satisfies $p = q^2$, where $q=010$. Thus, $\pi$ is collapsed, and Theorem~\ref{negativeshift} says that $\N(\pi)=3$.
In Lemma \ref{beginszeta}, we show that if $w$ induces $\pi$, then $w = \zeta w_{[n, \infty)}$. To get some intuition behind the theorem, let us see why the binary alphabet is not enough to realize $\pi$.  If $w$ were to induce $\pi$, then $w_{[y, \infty)} \less w_{[n, \infty)} \less w_{[x, \infty)}$ (where $x=1$ and $y=4$), that is,
\begin{equation}\label{eq:s'} 010w_{[n, \infty)} \less w_{[n, \infty)} \less 010010 w_{[n, \infty)},\end{equation}
which implies that $w_{[n, \infty)} = 010 w_{[n + 3, \infty)}$.  By the definition of $\less$, canceling the odd-length prefixes $010$ switches the inequality an odd number of times, and we find that
$$ 010 w_{[n + 3, \infty)} \gess w_{[n + 3, \infty)} \gess 010010w_{[n + 3, \infty)}.$$
But then $w_{[n + 3, \infty)}$ would have to start with $010$ as well. It follows from this argument that the only possibility would be $w_{[n, \infty)}=(010)^\infty$, which doesn't satisfy~\eqref{eq:s'}.  Thus, no word $w \in \mathcal{W}_2$ starting with $\zeta$ will induce the pattern $\pi$, and we must add an additional index to our segmentation in order to make it valid.  There are three valid $-3$-segmentations, giving rise to the words $\zeta^{(1)} = 121021$, $\zeta^{(2)} = 021020$, and $\zeta^{(3)} = 010020$.
\end{example}

The following two lemmas appear in~\cite{KArc} in the more general setting of signed shifts.

\begin{lem}  [\cite{KArc}] \label{onea}
Let $\zeta$ be the prefix corresponding to a segmentation of $\hat{\pi}$. If $\zeta$ can be completed to a word $w=\zeta w_{[n,\infty)}$ with $\Pat(w, \Sigma_{-N}, n) = \pi$, then the segmentation is valid. 
\end{lem}

\begin{proof}
Suppose for contradiction that $\zeta = w_{[1, n-1]}$ is such that $p = q^2$. Since $w$ induces $\pi$, we have $w_{[y, \infty)} \less w_{[n, \infty)} \less w_{[x, \infty)}$, or equivalently 
\begin{equation}\label{eq:q} q w_{[n, \infty)} \less w_{[n, \infty)} \less qq w_{[n, \infty)}.\end{equation}  

If $|q|$ is even, then canceling the even-length prefixes, $q$, gives $w_{[n, \infty)} \less q w_{[n, \infty)} = w_{[y, \infty)}$, which is impossible because $w$ induces $\pi$ and $\pi_{y} = \pi_{n} - 1$.

If $|q|$ is odd, Equation~\eqref{eq:q} implies that $w_{[n, \infty)} = q w_{[2n - y, \infty)}$.  Canceling prefixes of odd length we obtain $q w_{[2n - y, \infty)} \gess w_{[2n - y, \infty)} \gess qq w_{[2n-y, \infty)}$,
which implies that $w_{[2n -y, \infty)} $ must start with $q$ as well. Repeating this argument, it follows that the only possibility would be $w_{[n, \infty)} = q^\infty$, but this choice of $w_{[n, \infty)}$ doesn't satisfy~\eqref{eq:q}. 

An analogous argument shows that assuming $q =p^2$ also gives a contradiction.  Hence, the segmentation that produces $\zeta$ is valid.  
\end{proof}

\begin{lem} [\cite{KArc}] \label{beginszeta} If $w\in\WN$ and $\Pat(w, \Sigma_{-N}, n) = \pi$, then there exists a valid $-N$-segmentation of $\hat{\pi}$ whose associated prefix is $\zeta = w_{[1, n-1]}$. 
\end{lem}

\begin{proof}  Let $w \in \WN$ be such that $\Pat(w, \Sigma_{-N}, n ) = \pi$. For $0 \leq k \leq N$, let $e_k = |\{ 1 \leq r \leq n : w_{r} < k\}|$.  We claim that the sequence $0=e_0\le e_1\le\ldots\le e_N=n$ is a $-N$-segmentation of $\hat{\pi}$.  

First we show that condition (a) in Definition \ref{segmentation} holds.
By the definition of $e_k$, the prefix $w_{[1, n]}$ has $e_k$ letters less than $k$.  Therefore, among the subwords $w_{[r, \infty)}$ with $1 \leq r \leq n$, there are exactly $e_k$ of them with $w_r < k$, and
exactly $e_{k+1}$ of them with $w_r \le k$. Since $w$ induces $\pi$, it follows that if $e_k<\pi_i\le e_{k+1}$, then $w_{[i,\infty)}$ must be one of the subwords with $w_i \le k$ but not $w_i < k$, and so $w_i=k$. 

To show that the sequence $\hat{\pi}_{e_k +1}\hat{\pi}_{e_k +2}\dots \hat{\pi}_{e_{k+1}}$ is decreasing for all $0 \leq k < N$, 
suppose that $e_k < \pi_i < \pi_j \leq e_{k+1}$.  We will show that $\hat\pi_{\pi_i}>\hat\pi_{\pi_j}$ assuming that $i,j<n$, since the entry $\hat\pi_{\pi_n}=\star$ does not disrupt the property of $\hat{\pi}_{e_k +1}\hat{\pi}_{e_k +2}\dots \hat{\pi}_{e_{k+1}}$ being decreasing. 
By the previous paragraph, $w_i = w_j = k$, and $w_{[i, \infty)} \less w_{[j, \infty)}$ because $w$ induces $\pi$.
Therefore, $w_{[i + 1, \infty)} \gess w_{[j + 1, \infty)}$, and so $\pi_{i + 1} > \pi_{j + 1}$, or equivalently $\hat{\pi}_{\pi_i} = \pi_{i+1} > \pi_{j+1} = \hat{\pi}_{\pi_j}$.

To show that condition (b) holds, assume now that $\hat{\pi}_1 = n$ and $\hat{\pi}_{n-1} \hat{\pi}_{n} = 1 \star$, which is equivalent to $\pi_{n-2} \pi_{n-1} \pi_{n} = (n{-}1)1n$.  Suppose for contradiction that $e_1 > 0$ and $e_{N-1} < n-1$. Then, by the definition of 
the sequence $0 = e_0 \leq e_1 \leq \ldots \leq e_{N} = n$, $w_{[1, n]}$ has at least one $0$ and at least two $N{-}1$. 
Since $w$ induces $\pi$, we have that $w_{[n-1, \infty)}$ is the smallest and $w_{[n-2, \infty)}$ is the second largest among the subwords $w_{[r, \infty)}$ with $1 \leq r \leq n$. It follows that $w_{n-1} = 0$ and $w_{n-2} = N{-}1$. We cannot have $w_{[n, \infty)} = ((N{-}1)0)^\infty$, since then $w_{[n-2, \infty)} = w_{[n, \infty)}$ and $\Pat(w, \Sigma_{-N}, n)$ would be undefined.  Therefore, $w_{[n, \infty)} \less ((N{-}1)0)^\infty$, because $((N{-}1)0)^\infty$ is the largest word in $\WN$ with respect to $\less$.  It follows that $$w_{[n, \infty)} \less (N{-}1)0 w_{[n, \infty)} =w_{[n-2, \infty)},$$  contradicting that $w$ induces $\pi$ and $\pi_{n-2} < \pi_{n}$. 
Hence, condition (b) in Definition \ref{segmentation} holds.  

Verifying condition (c) follows a similar argument.  We conclude that $0 = e_0 \leq e_1 \leq \ldots  \leq e_{N} = n$ as defined above is a $-N$-segmentation of $\hat{\pi}$.  Its associated prefix is  $\zeta = w_{[1, n-1]}$ because we have seen that $w_i = k$ whenever $e_{k} < \pi_i \leq e_{k+1}$, which agrees with the construction of $\zeta$ in Definition \ref{segmentation}. Finally, since the prefix $\zeta$ can be completed to a word $w$ inducing $\pi$, it follows from Lemma~\ref{onea} that this $-N$-segmentation is valid.
\end{proof}

\begin{lem} \label{flipping} Let $\zeta$ be the prefix defined by some segmentation of $\hat\pi$, and let $i, j < n$. If $\pi_i < \pi_j$, then either $z_i < z_j$, or otherwise $z_i = z_j$ and $\pi_{i + 1} > \pi_{j+1}$.  \end{lem}

\begin{proof}  Suppose that $\pi_i < \pi_j$. Then the construction of $\zeta$ yields $z_i \leq z_j$.  We will prove that if $\pi_{i+1} < \pi_{j+1}$, then $z_i < z_j$.  By the definition of $\hat{\pi}$, we have $\hat{\pi}_{i} = \pi_{i + 1}$ and $\hat{\pi}_{j} = \pi_{j+1}$, and so $\hat{\pi}_{i} < \hat{\pi}_{j}$.  By Definition~\ref{segmentation}, the segmentation must contain an index $e_k$ such that $\hat{\pi}_i \leq e_k < \hat{\pi}_j$.  But then the construction of $\zeta$ then yields $z_i < z_j$. 
\end{proof}

\begin{lem} \label{geq}\label{numberseg}   A valid $-N$-segmentation of $\hat{\pi}$ exists if and only if $N \geq 1 + \asc(\hat{\pi}) + \epsilon(\hat{\pi})$. 
 Additionally, for $N=1 + \asc(\hat{\pi})+\epsilon(\hat{\pi})$, the number of distinct prefixes $\zeta$ arising from valid $-N$-segmentations of $\hat{\pi}$ is
\begin{itemize}
\item 1 if $\pi$ is regular;
\item 2 if $\pi$ is cornered;
\item $\min \{ |p|, |q| \}$ if $\pi$ is collapsed.
\end{itemize}
 \end{lem}

\begin{proof}  
We consider three cases, depending on whether $\pi$ is regular, cornered, or collapsed. 

If $\pi$ is regular, parts (b) and (c) of Definition~\ref{segmentation} do not apply, and so a $-N$-segmentation of $\hat{\pi}$ exists if and only if $N \geq 1 + \text{asc}(\hat{\pi})=1 + \asc(\hat{\pi}) + \epsilon(\hat{\pi})$. For  $N =1 + \asc(\hat{\pi})$, such a segmentation is a minimal segmentation, and thus it is valid (otherwise $\pi$ would be collapsed). For $N > 1 + \asc(\hat{\pi}) + \epsilon(\hat{\pi})$, one can obtain a valid $-N$-segmentation of $\hat{\pi}$ by adding indices $e_i = n$ for $1 + \asc(\hat{\pi}) < i \leq N$ to a minimal segmentation.  This reason it is valid is that the corresponding prefix is the same as the unique prefix $\zeta$ determined by a minimal segmentation of $\hat{\pi}$.  

If $\pi$ is cornered, then either part (b) or (c) of  Definition~\ref{segmentation} apply, requiring an additional index which is not an ascent of $\hat\pi$. Therefore, a $-N$-segmentation of $\hat{\pi}$ exists if and only if $N \geq 2 + \text{asc}(\hat{\pi}) =1 + \asc(\hat{\pi}) + \epsilon(\hat{\pi})$. Since a cornered permutation must have either $\pi_n = 1$ or $\pi_n = n$, one of the words $p$ and $q$ is not defined, and so any segmentation of $\hat\pi$ is valid. Additionally, for $N=2 + \text{asc}(\hat{\pi})$, if part (b) applies, we may choose either $e_1 = 0$ or $e_{N-1} \geq n-1$ as the additional index.  Whether we choose $e_{N-1} = n-1$ or $e_{N-1} = n$ does not change the associated prefix $\zeta$, since $\pi_n=n$ but the letter $z_n$ is not defined as a part of the prefix. A symmetric situation occurs when part (c) applies. In either case, there are two distinct prefixes $\zeta$ arising from a $-N$-segmentation of $\hat{\pi}$ when $N = 2 + \asc(\hat{\pi})$.  

If $\pi$ is collapsed, then the minimal segmentations of $\hat{\pi}$ are not valid. In order to obtain a valid segmentation, we must add an additional index.  Letting $c = \min\{ |p|, |q| \}$, the unique prefix $\zeta$ resulting from a minimal segmentation satisfies $z_{[n-2c, n-c-1]} = z_{[n - c, n-1]}$, and so we have $c$ pairs of equal letters, $z_{n-j} = z_{n-c-j}$ for $1\le j\le c$.  
If we add an index $e_k$ so that $\pi_{n-j} < e_{k} \leq \pi_{n-c-j}$ or $\pi_{n-j} < e_{k} \leq \pi_{n-c-j}$ (depending on the relative order of $\pi_{n-j}$ and $\pi_{n-c-j}$), then the corresponding prefix $\zeta'$ satisfies $z'_{n-j} \neq z'_{n-j-c}$.  This yields a valid $-(2 + \asc(\hat{\pi}))$-segmentation of $\hat{\pi}$, which can easily be extended to
a valid $-N$-segmentation for every $N\ge2+\asc(\hat\pi)$.

Let us show that, when $N=2 + \asc(\hat{\pi})$, there are exactly $c$ choices for the additional index $e_k$ that result into a valid $-N$-segmentation. 
We claim that, for $1\le j\le c$, the values $\pi_{n-j}$ and $\pi_{n-j-c}$ are consecutive.
Without loss of generality, let us assume that $\pi_{n-j}<\pi_{n-j-c}$, and  suppose for contradiction that there is an index $k$ such that $\pi_{n-j} < \pi_k < \pi_{n-j-c}$.  Since $z_{[n-j, n-1]} = z_{[n - j - c, n-c-1]}$, Lemma \ref{flipping} applied $k$ times yields $\pi_{n} < \pi_{k + j} < \pi_{n-c} = \pi_{x}$ or $\pi_{y} =\pi_{n-c} < \pi_{k + j} < \pi_{n}$ (depending on the parity of $j$), a contradiction to $\pi_x = \pi_n + 1$ or $\pi_{y} = \pi_n - 1$, respectively, thus proving the claim.
It follows that, for each $j$ with $1\le j\le c$, there is exactly one choice of $e_k$ satisfying $\pi_{n-j} < e_{k} \leq \pi_{n-c-j}$ or $\pi_{n-j} < e_{k} \leq \pi_{n-c-j}$, which forces $z'_{n-j} \neq z'_{n-j-c}$ in the associated prefix. This gives a total of $c$ choices for $e_k$. \end{proof}

Let us make make a few observations about the two prefixes that may arise from a $-N$-segmentation of $\hat\pi$ when $N = \asc(\hat{\pi}) +2$ when $\pi$ is cornered.  If $\pi_{n-2} \pi_{n-1} \pi_{n} = (n{-}1)1n$, then the $-N$-segmentation satisfies $e_1 = 0$ or $e_{N-1} \geq n-1$.  Choosing $e_{N-1} \geq n-1$ produces a prefix $\zeta \in \{0, 1, \dots, N{-}2 \}^{n-1}$.  Furthermore, since the indices $e_1,\dots,e_{N-2}$ must occur at the ascents of $\hat\pi$, $\zeta$ contains each letter in $\{0, \dots, N{-}2\}$.  Since $\pi_{n-1} = 1$, we get $z_{n-1} = 0$, and since $\pi_{n-2} = n-1$, we get $z_{n-2} = N-2$.  Hence, $q = (N{-}2)0$. On the other hand, choosing $e_1 = 0$ produces a prefix  $\zeta^+ \in \{1, 2, \dots, N{-}1\}^{n-1}$, and by the same logic we get $q = (N{-}1)1$. Similarly, if $\pi$ is cornered of the form $\pi_{n-2} \pi_{n-1} \pi_{n} = 2n1$, the two $-N$-segmentations of $\hat\pi$ have associated prefixes $\zeta \in \{0, 1, ..., N{-}2\}^{n-1}$ with $p = 0(N{-}2)$ and $\zeta^+ \in \{1, 2, \dots, N{-}1 \}^{n-1}$ with $p = 1 (N{-}1)$.

For the next five lemmas (from Lemma~\ref{d2} to Lemma~\ref{pattern}), fix $N = 1 + \asc(\hat{\pi}) + \epsilon(\hat{\pi})$, and let $\zeta$ be a prefix determined by some valid $-N$-segmentation of $\hat{\pi}$, guaranteed to exist by Lemma \ref{geq}. 

We include the proof of the following lemma from~\cite{KArc} to make this section self-contained.

\begin{lem}  [\cite{KArc}] \label{d2} With $\zeta$ as defined above, either $p$ is primitive, or $p = d^2$, where $d$ is primitive and $|d|$ is odd.  The same is true for $q$.
\end{lem}

\begin{proof} We can write $p = d^r$, where $d$ primitive, and let $i=|d|$.  Then $n = x+ri$ and $$d=z_{[x, x+i-1]} = z_{[x + i, x + 2i - 1]} = \dots = z_{[x + (r-1)i, n-1]}.$$  

Suppose first that $i$ is even. If $\pi_x < \pi_{x + i}$, then applying Lemma \ref{flipping} $i$ times we obtain $\pi_{x+i} < \pi_{x + 2i}$.  Repeatedly applying this argument yields
$$\pi_x < \pi_{x+i} < \pi_{x + 2i} < \dots < \pi_{x + ri} = \pi_n,$$
which contradicts the fact that $\pi_x = \pi_n + 1$.  On the other hand, if $\pi_x > \pi_{x + i}$, then we get
$$\pi_x > \pi_{x+i} > \pi_{x + 2i} > \dots > \pi_{x + ri} = \pi_n.$$  Since $\pi_x = \pi_n + 1$, we must have $r = 1$, and so $p$ must be primitive in this case. 

Now suppose that $i$ is odd.  If $r$ is even, then we can write $p=(d')^{r/2}$ with $d' = d^2$ and apply the previous argument (which does not require $d'$ to be primitive) to conclude that $r/2=1$ and $p=d^2$, with $|d|=i$ odd. We are left we the case that $r$ is odd.  

If $\pi_x < \pi_{x + i}$, then Lemma \ref{flipping} applied $i$ times implies that $\pi_{x + i} > \pi_{x + 2i}$.  Consider two cases depending on the relative order of $\pi_x$ and $\pi_{x+2i}$.
If $\pi_x < \pi_{x+2i} < \pi_{x + i}$, then applying Lemma~\ref{flipping} $i$ times gives $\pi_{x+i} > \pi_{x + 3i} > \pi_{x + 2i}$.  Applying the same lemma $i$ more times we obtain $\pi_{x + 2i} < \pi_{x + 4i} < \pi_{x + 3i}$. Repeated applications of Lemma~\ref{flipping} give
$$\pi_x < \pi_{x + 2i} < \pi_{x + 4i} < \dots < \pi_{x + (r-1)i} < \pi_{x + ri} < \pi_{x + (r-2)i} < \dots < \pi_{x + 3i} < \pi_{x + i}.$$
Similarly, if $\pi_{x + 2i} < \pi_x < \pi_{x + i}$, repeated applications of Lemma~\ref{flipping} give
$$\pi_{x + (r-1)i} < \dots < \pi_{x + 4i} < \pi_{x + 2i} < \pi_x < \pi_{x + i} < \pi_{x + 3i} < \dots < \pi_{x + ri},$$
In both cases, we get $\pi_x < \pi_{x + ri} = \pi_n$, a contradiction to $\pi_x = \pi_n + 1$.  

If $\pi_x > \pi_{x + i}$, then Lemma \ref{flipping} applied $i$ times implies that $\pi_{x + i} < \pi_{x + 2i}$.  Again, we consider two cases
depending on the relative order of $\pi_x$ and $\pi_{x+2i}$. If $\pi_{x + i} < \pi_x < \pi_{x +2i}$, then repeated applications of Lemma \ref{flipping} give
$$\pi_{x + ri} < \dots < \pi_{x + 3i} < \pi_{x + i} < \pi_x < \pi_{x + 2i} < \pi_{ x + 4i} < \dots < \pi_{x + (r-1)i}.$$
Similarly, if $\pi_{x +i} < \pi_{x + 2i} < \pi_x$, then Lemma \ref{flipping} gives
$$\pi_{x + i} < \pi_{x + 3i} < \dots < \pi_{x + ri} < \pi_{x + (r-1)i} < \dots < \pi_{x + 4i} < \pi_{x + 2i} < \pi_x.$$
In both cases, the fact that $\pi_x = \pi_n + 1= \pi_{x + ri}+1$ implies that $r = 1$, and so $p$ is primitive.

The proof that $q$ is either primitive or the square of a primitive word of odd length follows a parallel argument.  
\end{proof}

It follows from Lemma~\ref{d2} that if $p = q^2$, then $q$ is primitive and $|q|$ is odd.  Likewise, if $q = p^2$, then $p$ is primitive and $|p|$ is odd.  

Note that $\wmaxn$ and $\wminn$ are the largest and the smallest words in $\WN$, respectively, with respect to $\less$. When $\pi_n\neq n$ (so that $x$ and $p$ are defined), let
\begin{equation}\label{eq:defs} s =    \begin{cases}
     \zeta p^{n-2} \wminn & \text{ if $n$ is even or $|p|$ is even,}\\
          \zeta p^{n-2} \wmaxn & \text{ if $n$ is odd and $|p|$ is odd.}
\end{cases} \end{equation}
Similarly, if $\pi_n\neq 1$ (so that $y$ and $q$ are defined), let
\begin{equation} t =    \begin{cases}
     \zeta q^{n-2} \wmaxn & \text{ if $n$ is even or $|q|$ is even,}\\
          \zeta q^{n-2} \wminn & \text{ if $n$ is odd and $|q|$ is odd. }
\end{cases} \end{equation} 
Note that $s, t\in\WN$ by construction. We will show that $s$ and $t$ induce $\pi$.

\begin{lem} \label{onecollapsed} If $\zeta = a qq$ (for some $a$) and $|q|$ is odd, then $p =q^2$ (in particular, $p$ is defined).  Likewise, if $\zeta = a' pp$ (for some $a'$) and $|p|$ is odd, then $q = p^2$.  \end{lem}

\begin{proof}  Let $i = |q|$ and $m = n - 2i = y - i$. Then $z_{[m, y-1]} = z_{[y, n-1]} =q$, which is primitive by  Lemma~\ref{d2} because $|q|$ is odd.  By the contrapositive of Lemma \ref{flipping} applied $i$ times, $\pi_y < \pi_n$ implies that $\pi_{m} > \pi_y$.  Since $\pi_y=\pi_n-1$, we have $\pi_y < \pi_n <  \pi_m$.  We will show that $\pi_m = \pi_n +1$, from where it will follow that $m=x$ and $p = q^2$. Suppose for contradiction that there exists some $\pi_k$ such that $\pi_n < \pi_k < \pi_m$.

Consider first the case $k < y$. Since $\pi_y < \pi_k < \pi_m$, Lemma \ref{flipping} and the fact that $z_y=z_m$ forces $z_y=z_k=z_m$ and $\pi_{y+1} > \pi_{k+1} > \pi_{m+1}$. Applying the same argument $i$ times yields $z_{[y, n-1]}=z_{[k, k+i - 1]}= z_{[m, y-1]} =q$ and $\pi_n=\pi_{y+i} > \pi_{k+i} > \pi_{m+i}=\pi_y$. But the fact that  $\pi_n> \pi_{k+i} > \pi_y$ contradicts $\pi_y = \pi_n - 1$.

Consider now the case $k > y$. 
Since $\pi_y < \pi_k < \pi_m$ and $z_{[y, m-1]} = z_{[m, n-1]} = q$, Lemma \ref{flipping} applied $n-k$ times implies that $z_{[y, y + n-k - 1]} = z_{[k, n-1]} = z_{[m, m + n - k - 1]}$. Additionally, if $n -k$ is odd, it yields $\pi_{y + n-k} > \pi_n > \pi_{m + n-k}$, and so $\pi_{y + n - k} > \pi_y > \pi_{m + n -k}$. Applying Lemma \ref{flipping} $k-y$ more times, we get $z_{[y + n - k, n -1]} = z_{[y, k-1]} = z_{[m + n - k, y - 1]}$.  Similarly, if $n - k$ is even, we first get $\pi_{y + n - k} < \pi_n < \pi_{m + n - k}$ and $\pi_{y + n - k} < \pi_y < \pi_{m + n -k}$, and then using Lemma \ref{flipping} again we conclude that $z_{[y + n - k, n -1]} = z_{[y, k-1]} = z_{[m + n - k, y - 1]}$ as well.
Combining the above equalities, we have $z_{[k, n-1]}z_{[y, k-1]}= z_{[y, y + n-k - 1]} z_{[y + n - k, n-1]}=z_{[y, n-1]}=q$, which states that $q$ is equal to one of its non-trivial cyclic shifts, thus contradicting that it is primitive. 

\end{proof}

\begin{lem} \label{twopointnine}  Let $w\in\{s,t\}$ and suppose it is defined. Then $\Pat(w,\Sigma_{-N},n)$ is defined as well. \end{lem}  

\begin{proof} We prove the statement for $w=s$. The proof for $w=t$ is analogous.

Suppose first that $p\neq 0(N{-}1)$. Note that by Lemma~\ref{d2}, we also have $p\neq (0(N{-}1))^r$ for all $r\ge2$.
Thus, for $i, j \leq n$, the equality $w_{[i, \infty)} = w_{[j, \infty)}$ implies that these two words have the first instance of $\wminn$ appearing at the same position, forcing $i=j$. Therefore, $\Pat(w, \Sigma_{-N}, n)$ is defined.  

Suppose now that $p = 0(N{-}1)$. Note that $x=n-2$ and $\pi_{n-2}=\pi_n+1$ in this case. If there is an index $i < n-2$ such that $\pi_i < \pi_{n-2}$, take the maximal one. Since $z_{n-2} = 0$, Lemma \ref{flipping} implies that $z_{i} = z_{n-2} = 0$ and $\pi_{i + 1} > \pi_{n-1}$. Similarly, since $z_{n-1} = N{-}1$, applying Lemma \ref{flipping} again gives  $z_{i+1} = z_{n-1} = N{-}1$ and $\pi_{i + 2} < \pi_{n} < \pi_{n-2}$, contradicting the maximality of $i$. It follows that $\pi_i >\pi_{n-2}$ for all $i < n-2$. Since clearly $\pi_{n-2} < \pi_{n-1}$ because $z_{n-1} = N{-}1$, we conclude that $\pi_{n-2} = 2$ and $\pi_n = 1$.  Now, if there was an index $j$ such that $\pi_j > \pi_{n-1}$, then Lemma \ref{flipping} would give $z_{j} = z_{n-1} = N{-}1$ and $\pi_{j+1} < \pi_{n} + 1=2$, which is impossible. We conclude that $\pi_{n-1} = n$. 

We have shown that in the case $p = 0(N{-}1)$, we must have $\pi_{n-2} \pi_{n-1} \pi_n = 2 n 1$, and so $\pi$ is cornered. By part (c) of  Definition \ref{segmentation}, a $-N$-segmentation of $\hat{\pi}$ has either $e_{N-1}  = n$ or $e_1 \leq 1$.   If $e_{N-1} = n$, then $\zeta$ does not contain the letter $N{-}1$ by construction.  Likewise, if $e_1 \leq 1$, then $\zeta$ does not contain the letter $0$ because the only entry of $\pi$ that can satisfy $\pi_i\le e_1$ is $\pi_n=1$. Thus, $\zeta$ cannot contain both a $0$ and a $N{-}1$, which contradicts that $p = 0(N{-}1)$. 
\end{proof}

\begin{lem}  \label{nothingbetween}  For the word $s$, we have $s_{[n, \infty)} \less s_{[x, \infty)}$ and there is no $1 \leq c \leq n$ such that $s_{[n, \infty)} \less s_{[c, \infty)} \less s_{[x, \infty)}$.   Likewise, $t_{[y, \infty)} \less t_{[n, \infty)}$ and there is no $1 \leq c \leq n$ such that $t_{[y, \infty)}$ $\less t_{[c, \infty)} \less t_{[n, \infty)}$.  \end{lem}  

\begin{proof} 
We will prove the statement for $s$. The one for $t$ is analogous. The fact that $s_{[n, \infty)} \less s_{[x, \infty)}$ follows immediately by canceling equal prefixes in the word. Indeed, if $n$ is even or $|p|$ is even, this is equivalent to $p^{n-2} \wminn \less p^{n-1} \wminn$, and to  $\wminn \less p \wminn$, which holds because $\wminn$ is the smallest word in $\WN$ with respect to $\less$.  If both $n$ and $|p|$ are odd,  $s_{[n, \infty)} \less s_{[x, \infty)}$ is equivalent to $\wmaxn \gess p \wmaxn$, which again holds because $\wmaxn$ is the largest word in $\WN$ with respect to $\gess$.  

Next we prove that there is no $1 \leq c \leq n$ such that $s_{[n, \infty)} \less s_{[c, \infty)} \less s_{[x, \infty)}$, that is,
$$p^{n-2} \wminn \less s_{[c, \infty)} \less p^{n-1} \wminn.$$ 
Suppose for contradiction that such a $c$ existed. Then $s_{[c, \infty)} = p^{n-2}v$ for some word $v$ satisfying $\wminn \less v \less p \wminn$ (if $n$ or $|p|$ are even)
or  $\wminn \gess v \gess p \wminn$ (if $n$ and $|p|$ are odd).

We claim that $c<x$. If $p$ is primitive, this is because the first $p$ in $s_{[c, \infty)}$ cannot overlap with both the first and second occurrences of $p$ in $s_{[x, \infty)}$. If $p$ is not primitive, then by Lemma~\ref{d2}, $p = d^2$  where $d$ is primitive and $|d|$ is odd. The only way to have $c> x$ would be if $v = d\wminn$, the largest word beginning with $d$, but this is impossible because $v \less d^2 \wminn$.  

Next we show that $v$ begins with a $p$. Consider first the case when $p$ is primitive.  
Unless $|p| = 1$, $c=1$ and $x = n-1$, one of the initial $n-2$ occurrences of $p$ in $s_{[c, \infty)}$ must coincide with the first occurrence of $p$ in $s_{[x, \infty)}$, since $|p^{n-2}|> n-1$, and so $v$ begins with $p$. If $|p| = 1$, $c=1$ and $x = n-1$, we have $s_{[c, \infty)} = p^{n-2} s_{[x, \infty)}$, and since $s_{[x, \infty)}$ begins with a $p$, we have that $v$ begins with a $p$ as well.  
If $p$ is not primitive, then $p = d^2$ where $d$ is primitive and $|d|$ is odd, by Lemma~\ref{d2}.  Since $|d^{2(n-2)}| > n -1$, one of the initial $2(n-2)$ occurrences of $d$ in $s_{[c, \infty)}$ must coincide with the first occurrence of $d$ in $s_{[x, \infty)}$.  The fact that $s_{[x, \infty)}$ begins with $d^{2(n-1)}$ and $c < x$ implies that $v$ begins with $d^2 = p$.

If $|p|$ is even, the fact that $v$ begins with a $p$ contradicts that $v \less p \wminn$, since $p \wminn$ is the smallest word beginning with $p$.  If $|p|$ is odd, then the above argument causes $\zeta$ to be of the form $\zeta = app$ for some $a$.  By Lemma \ref{onecollapsed}, this implies that $q=p^2$, contradicting the fact that $\zeta$ was obtained from a valid $-N$-segmentation. 

The proof for $t$ follows in a similar fashion.   \end{proof}

\begin{lem}\label{pattern} Let $w = \zeta w_{[n, \infty)} \in \WN$ be such that $\Pat(w, \Sigma_{-N}, n)$ is defined.
 If $w_{[x, \infty)} \gess w_{[n, \infty)}$ and there is no $1 \leq c \leq n$ such that $w_{[n, \infty)} \less w_{[c, \infty)} \less w_{[x, \infty)}$, then $\Pat(w, \Sigma_{-N}, n) = \pi$. Likewise, if $w_{[y, \infty)} \less w_{[n, \infty)}$ and there is no $1 \leq c \leq n$ such that $w_{[y, \infty)} \less w_{[c, \infty)} \less w_{[n, \infty)}$, then $\Pat(w,  \Sigma_{-N}, n) = \pi$.    \end{lem}

\begin{proof} We prove the statement for $w_{[x, \infty)}$. The one involving $w_{[y, \infty)}$ follows similarly. For $1 \leq i, j \leq n$, let $S(i, j)$ be the statement 
$$\pi_i < \pi_j \text{ implies } w_{[i, \infty)} \less w_{[j, \infty)}.$$
To show that  $\Pat(w,  \Sigma_{-N}, n) = \pi$, we will prove $S(i, j)$ for all $1 \leq i, j \leq n$ with $i \neq j$.  We consider three cases.  

\begin{enumerate}
\item Case $i = n$.  Suppose that $\pi_n < \pi_j$.  By assumption, $w_{[n, \infty)} \less w_{[x, \infty)}$.  If $j = x$, we are done.  If $j \neq x$, then $\pi_n < \pi_j$  implies that $\pi_x < \pi_j$ since $\pi_x = \pi_n + 1$.  So, if $S(x, j)$ holds, then $w_{[n, \infty)} \less w_{[x, \infty)} \less w_{[j, \infty)}$, so $S(n, j)$ must hold as well.  We have reduced $S(n, j)$ to $S(x, j)$.  Equivalently, $\neg S(n, j) \rightarrow \neg S(x, j)$, where $\neg$ denotes negation.  

\item Case $j = n$.  Suppose that $\pi_i < \pi_n$.  In particular,  $i \neq n$ and  $\pi_i < \pi_x=\pi_n+1$. By assumption, in order to prove that $w_{[i, \infty)} \less w_{[n, \infty)}$, it is enough to show that $w_{[i, \infty)} \less w_{[x, \infty)}$.   Thus, we have reduced $S(i, n)$ to $S(i, x)$. 

\item Case $i, j < n$.  Suppose that $\pi_i < \pi_j$.  Let $m$ be so that $w_{[i, i + m -1]} = w_{[j, j + m -1]}$ and $w_{i + m} \neq w_{j + m}$.  First assume that $i + m, j + m \leq n -1$.  If $m$ is even, then Lemma \ref{flipping} applied $m$ times to $\pi_{i} < \pi_{j}$ implies that $\pi_{i + m} < \pi_{j + m}$.  By Lemma \ref{flipping}, we must have $w_{i + m} \leq w_{j + m}$, and so we conclude that $w_{i + m} < w_{j + m}$.  Therefore, $w_{[i + m, \infty)} \less w_{[j + m, \infty)}$, and thus $w_{[i, \infty)} \less w_{[j, \infty)}$.
Similarly, if $m$ is odd, Lemma \ref{flipping} applied $m$ times implies that $\pi_{i + m} > \pi_{j + m}$.  Hence, by Lemma \ref{flipping} we must have $w_{i + m} > w_{j + m}$ because $w_{i + m} \neq w_{j + m}$.  Therefore, $w_{[i + m, \infty)} \gess w_{[j + m, \infty)}$, and thus $w_{[i, \infty)} \less w_{[j, \infty)}$ again. This shows that if $i + m, j + m \leq n-1$, then $S(i,j)$ holds.

Suppose now that $i + m \geq n$ or $j+m \geq n$, and let $m'$ be the minimal index such that either $i + m' = n$ or $j + m' = n$.  Suppose first that $i + m' = n$ and $m'$ is even. We claim that $S(i, j)$ reduces to $S(n, j + m')$ in this case. Indeed, suppose that $S(n, j + m')$, and let us show that $S(i,j)$ holds as well. If $\pi_i < \pi_j$, then Lemma \ref{flipping} and the fact that $w_{[i, i + m' -1]} = w_{[j, j + m' - 1]}$ gives $\pi_n = \pi_{i + m'} < \pi_{j + m'}$.  Since $S(n, j +m')$ holds, we have $w_{[n, \infty)} \less w_{[j + m', \infty)}$, which implies that $w_{[i, \infty)} \less w_{[j, \infty)}$, as desired.  Thus, $S(i, j)$ reduces to $S(n, j + m')$.  

Similarly, if $i + m' = n$ and $m'$ is odd, then $\pi_i < \pi_j$ and $w_{[i, i + m' - 1]} = w_{[j, j + m' - 1]}$ implies that $\pi_{n} = \pi_{i + m'} > \pi_{j + m'}$ by Lemma \ref{flipping}. If $S(j +m',n)$ holds, then $w_{[n, \infty)} \gess w_{[j + m', \infty)}$, which implies that $w_{[i, \infty)} \less w_{[j, \infty)}$. Again, $S(i, j)$ reduces to $S(j + m', n)$ in this case.  

Now consider the case when $j + m' = n$ and $m'$ is odd. Then $\pi_i < \pi_j$ and $w_{[i, i + m' -1]} = w_{[j, j + m' - 1]}$ implies that $\pi_{i + m'} > \pi_{j + m'} = \pi_n$ by Lemma \ref{flipping}.  Therefore, $S(i, j)$ reduces to $S(n, i + m')$ in this case.  Finally, if $j + m' = n$ and $m'$ is even, $S(i, j)$ reduces to $S(i + m', n)$ by a similar argument.  
\end{enumerate}

In order to conclude that $S(i, j)$ holds for every $i,j$, we must show that the above process of reductions eventually terminates. Suppose for contradiction that the process goes on indefinitely.  Then at some point we would reach $S(x, k)$ with $k > x$, or $S(k, x)$ with $k > x$.    

\begin{itemize}
\item Suppose that we reach $S(x, k)$ with $k > x$. Since we assumed that the process does not terminate, case (3) above implies that $w_{[x, x+ n-k -1]} = w_{[k, n - 1]}$.
If $n-k$ is odd, then we get $\neg S(x,k) \rightarrow \neg S(n, x + n-k) \rightarrow \neg  S(x, x + n-k)$, using cases (3) and (1). If $n-k$ is even, then $\neg S(x,k) \rightarrow \neg S(x + n-k,n) \rightarrow \neg  S(x + n-k,x)$, using cases (3) and (2). 
\item  Suppose we reach $S(k, x)$ with $k > x$.
Since the process does not terminate, case (3) implies that $w_{[k, n - 1]}=w_{[x, x+ n-k -1]}$.
If $n-k$ is odd, then we get $\neg S(k,x) \rightarrow \neg S(x + n-k,n) \rightarrow \neg  S(x + n-k,x)$. If $n-k$ is even, then $\neg S(k,x) \rightarrow \neg S(n,x + n-k) \rightarrow \neg  S(x,x + n-k)$.
\end{itemize}

In all cases, we conclude that $w_{[x, x+ n-k -1]} = w_{[k, n - 1]}$ and we reach $S(x,x + n-k)$ or $S(x + n-k,x)$. Now we can repeat the argument with $x+n-k$ playing the role of $k$, to deduce that $w_{[x, k -1]} = w_{[x+n-k, n - 1]}$ and obtain a reduction back to $S(x,k)$ or $S(k,x)$.

Combining the above equalities, we obtain $w_{[x, x + n - k - 1]} w_{[x + n -k - 1, n -1]} = w_{[x, k-1]} w_{[k, n-1]}  = p$.   Since $p$ is equal to some of its non-trivial cyclic shifts, it follows that $p$ is not primitive.  Thus, in the case that $p$ is primitive, we have verified the statements $S(x, k)$ and $S(k, x)$.  

Therefore, by Lemma \ref{d2}, we have $p = d^2$, where $d = w_{[x, k - 1]} = w_{[k, n-1]}$ and $|d|$ is odd.  Let $d = w_{[x, k -1]}$, so that we may write $p = d^2$.   It remains to verify the statements $S(x, k)$ and $S(k, x)$ in this case. 
 
To verify $S(k, x)$, suppose that $\pi_k < \pi_x$.  Since $\pi_n = \pi_x - 1$ and $k \neq n$, we have $\pi_{k} < \pi_n < \pi_x$.   We claim that, in this case, $\pi_{k} = \pi_y$.  
Suppose for contradiction that  $\pi_{k} \neq \pi_{n} - 1= \pi_y$.  Let $1 \leq h < n$  be the largest index such that $\pi_{k} < \pi_{h} < \pi_x$.

Consider first the case $h > k$. Since $w_{[k, n-1]} = w_{[x, k-1]}$, Lemma \ref{flipping} applied $n-k$ times to $\pi_k < \pi_x$ implies that \begin{equation}\label{eq:1}w_{[k, k + n - h - 1]} = w_{[h, n-1]} = w_{[x, x + n - h -1]}.\end{equation}   If $n - h$ is odd, then Lemma \ref{flipping} gives $\pi_{k + n - h} > \pi_{n} > \pi_{x + n - h}$, and so $\pi_{k + n -h} > \pi_x > \pi_{x + n - h}$ as well. Applying Lemma \ref{flipping} $h-k$ more times, we conclude that \begin{equation}\label{eq:2}w_{[k + n - h, n-1]} = w_{[x, x+ h - k - 1]}= w_{[k, h-1]},\end{equation} where in the last equality we used that $w_{[k, n-1]} = w_{[x, k-1]}$.
On the other hand, if $n - h$ is even, we get $\pi_{k + n - h} < \pi_{n} < \pi_{x + n - h}$ and $\pi_{k + n -h} < \pi_x < \pi_{x + n - h}$, from where Equation~\eqref{eq:2} holds as well. Combining Equations~\eqref{eq:1} and~\eqref{eq:2}, we get $d = w_{[k, n-1]} = w_{[k, h - 1]} w_{[h, n-1]} =  w_{[k + n - h, n -1]} w_{[k, k + n - h -1]}$, which states that $d$ is equal to one of its non-trivial cyclic shifts, thus contradicting that it is primitive.  

Now consider the case $h < k$.  Applying Lemma \ref{flipping} $|d|$ times to the inequalities $\pi_k < \pi_{h} < \pi_x$, we obtain $\pi_n > \pi_{h + k - x} > \pi_{k}$.  Therefore, $\pi_x > \pi_{h + k - x} > \pi_{k}$, a contradiction to the fact that we chose $h$ to be the largest index such that $\pi_{k} < \pi_{h} < \pi_x$. 

It follows that there is no index $h \neq n$ such that $\pi_{k} < \pi_{h} < \pi_x$.  We conclude that $\pi_{k} = \pi_y$, from which it follows that $d=q$ and $p = q^2$ .  However, this contradicts that $\zeta$ comes from a valid $-N$-segmentation of $\hat{\pi}$.
Since the assumption $\pi_k < \pi_x$ leads to a contradiction, the statement $S(k,x)$ trivially holds.

To verify $S(x, k)$, suppose now that $\pi_x < \pi_k$. We must show that $w_{[x, \infty)} \less w_{[k, \infty)}$.  Suppose to the contrary that $w_{[k, \infty)} \less w_{[x, \infty)}$.  Then, by assumption, we must have $w_{[k, \infty)} \less w_{[n, \infty)} \less w_{[x, \infty)}$.  Hence,
\begin{equation}\label{eq:ineqd}
d^2 w_{[n, \infty)} \less w_{[n, \infty)} \less d w_{[n, \infty)}.
\end{equation}
Therefore, $w_{[n, \infty)}$ must begin with $d$, and by canceling equal prefixes, we determine that the only option would be $w_{[n, \infty)} = d^\infty$, which does not satisfy the inequalities~\eqref{eq:ineqd}.  We conclude that $w_{[x, \infty)} \less w_{[k, \infty)}$, and so $S(x, k)$ holds.

We have shown that if the above process of reductions does not terminate, then it reaches $S(x,k)$ or $S(k,x)$ where $k-x=n-k$ and $p=d^2$. And we have shown that both $S(x,k)$ or $S(k,x)$ hold in this case. It follows that $S(i, j)$ holds for all $1 \leq i, j \leq n$ with $i \neq j$.  
\end{proof}

\begin{proof} [Proof of Theorem \ref{negativeshift}] 
We will show that $\pi \in \Al(\Sigma_{-N})$ if and only if $N \geq 1 + \asc(\hat{\pi}) + \epsilon(\hat{\pi})$. 

Suppose first that $\pi \in \Al(\Sigma_{-N})$. By Lemma \ref{onea}, $\hat{\pi}$ has a valid $-N$-segmentation. By Lemma \ref{geq}, such a valid segmentation exists if and only if $N \geq 1 + \asc(\hat{\pi}) + \epsilon(\hat{\pi})$.  Therefore, $\pi \in \Al(\Sigma_{-N})$ implies that $N \geq 1 + \asc(\hat{\pi}) + \epsilon(\hat{\pi})$.

For the other direction, by Lemma~\ref{lem:monotoneN}, it is enough to show that if we let $N=1 + \asc(\hat{\pi}) + \epsilon(\hat{\pi})$,
then $\pi \in \Al(\Sigma_{-N})$. Right before Lemma~\ref{onecollapsed}, we construct words $s, t\in\WN$ (at least one of which is always defined), and in Lemmas \ref{twopointnine}, \ref{nothingbetween} and \ref{pattern} we show that they induce $\pi$.
\end{proof}

In~\cite{EMupcoming} we use this analysis to count the number of permutations of length $n$ realized by $\Sigma_{-N}$, and we apply similar arguments to signed shifts, obtaining bounds on the number of patterns realized by the tent map.

\section{$-\beta$-Expansions}\label{sec:beta}

For any $\beta > 1$, the {\em $-\beta$-expansion} of $x \in (0, 1]$ is the sequence $\varepsilon_1(x)\varepsilon_2(x)\dots$ defined by $\varepsilon_i(x) = \lfloor \beta \Tb^{i-1}(x) \rfloor$, with $\Tb$ given by Equation~\eqref{eq:Tb}. It satisfies
$$x = - \sum_{i = 1}^\infty \frac{\varepsilon_i(x) + 1}{(-\beta)^i}.$$ Throughout this section, let $N = \lfloor \beta \rfloor +1$ and note that $\varepsilon_i(x) \in \{0, 1, \dots N{-}1 \}$ for all $i$.

Let $\Wb^0\subseteq\WN$ be the set of $-\beta$-expansions of numbers in $(0,1]$, and let $\betaone = a_1 a_2 a_3 \dots$ denote the $-\beta$-expansion of $1$. Ito and Sadahiro~\cite{ItoS} characterized the set  $\Wb^0$ as follows.

\begin{thm} [\cite{ItoS}] \label{thm:Wb} If  $\betaone$ is not periodic of odd length, then
$$\Wb^0 = \{ w : 0\betaone  \less w_{[k,\infty)}  \lesseq \betaone \text{ for all } k \geq 1 \}.$$
If $\betaone = (a_1 a_2 \dots a_{2r + 1})^\infty$ 
for some $r \geq 0$, and $r$ is minimal with this property, then 
$$\Wb^0  = \{ w : (0a_1 \dots a_{2r}(a_{2r+1} - 1))^\infty \less w_{[k,\infty)} \lesseq \betaone \text{ for all } k \geq 1\}.$$ \end{thm}

It follows from the above theorem that if $w \in \Wb^0$, then $w_{[k, \infty)} \in \Wb^0$ for any $k \geq 1$.  In particular, shifts of $\betaone$ satisfy ${\mathfrak{a}_{\beta}}_{[k, \infty)} \leq \betaone$ for all $k \geq 1$.

Given an infinite word $w = w_1 w_2 \dots\in\WN$, define the series
$$ f_w(\beta) = - \sum_{j=1}^\infty \frac{w_j+1}{(-\beta)^j}.$$
Note that $f_w(\beta)$ is convergent for $\beta >1$.  It is shown in~\cite{ItoS} that this map is order-preserving in the following sense.

\begin{lem} [\cite{ItoS}] \label{orderinWb0} Let $v, w \in \Wb^0$.  If $v \less w$, then $f_v(\beta) < f_w(\beta)$.
\end{lem}

If $w\in\Wb^0$ is the $-\beta$-expansion of $x\in(0,1]$, then $f_w(\beta)=x$, and so the inverse of the map 
\begin{equation}\label{eq:fw} \Wb^0 \rightarrow (0, 1], \quad w\mapsto f_{w}(\beta) \end{equation}
is the map that associates each $x\in(0,1]$ to its $-\beta$-expansion $\varepsilon_1(x)\varepsilon_2(x)\dots$.

In terms of words, the {\em negative $\beta$-shift} is defined as the map
$$\Sb: \Wb^0 \to \Wb^0, \quad w_1w_2w_3\dots \to w_2w_3\dots,$$
with the order $\less$ on $\Wb^0$. We will write $\Sigma_{-}$ when we do not need to specify the domain.

\begin{lem} [\cite{ItoS}] \label{itos} The map $\Sb$ on $(\Wb^0, \less)$ and the map $\Tb$ on $((0, 1], < )$ are order-isomorphic, via the order-isomorphism in Equation~\eqref{eq:fw}.  \end{lem}

It will be convenient to define $\Sb$ in a larger domain $\Wb\supseteq\Wb^0$, as follows. 

\begin{defn}\label{def:Wb} Let 
$$\Wb = \{ w\in\WN : 0\leq  f_{w_{[k, \infty)}}(\beta) \leq 1\text{ for all } k \geq 1 \}.$$
Moreover, define $\wmax$ and $\wmin$ to be the largest and the smallest words in $\Wb$ with respect to $\less$, respectively. 
\end{defn}

By the above definition, if a word $w$ is in $\Wb$, then so are all its shifts $w_{[k, \infty)}$ for $k\ge1$.
In the rest of the paper we consider $\Wb$ to be the domain of $\Sb$. Thus, we define $$\Al(\Sb)=\bigcup_{n\ge0}\{\Pat(w,\Sb,n):w\in\Wb\}.$$ This choice of domain, which will simplify some of our proofs, does not affect our results about the smallest $\beta$ needed to realize a pattern,
as shown in Proposition~\ref{domain}.

Since $w \lesseq \wmax$ for all $w \in \Wb$ by definition, we have that $0 \wmax \lesseq w$ for all $w \in \Wb$.  Therefore, $\wmin = 0 \wmax$ is the smallest word in $\Wb$.

In the case that $\beta = K$ is an integer, the $-K$-expansion of $1$ is $K^\infty$,  and so $\Omega_K = K^\infty$ and $\omega_K = 0 K^\infty$.  In particular, $\mathcal{W}_K \subsetneq \mathcal{W}_{-K}$. 
This discrepancy  is a result of defining the reverse shift in Section~\ref{sec:-N} to agree with the definition of signed shifts from~\cite{Amigosigned,ArcEli}, while defining the negative $\beta$-shift according to the constructions in \cite{ItoS,Steiner} in order to be able to apply the results in these papers. Next we show that the allowed patterns of $\Sigma_{-K}$ are the same regardless of whether we take $\mathcal{W}^0_{K}$ or $\mathcal{W}^0_{-K}$ to be its domain.

\begin{lem}
In the case that $\beta = K \geq 2$ is an integer, 
$$\Al(\Sigma_{-K}|_{\mathcal{W}^0_{K}}) = \Al(\Sigma_{-K}|_{\mathcal{W}^0_{-K}}).$$ 
\end{lem}

\begin{proof}  The $-K$-expansion of $1$ is $\mathfrak{a}_{K} = K^\infty$, and in this case Theorem \ref{thm:Wb} states that 
$$\mathcal{W}_{-K}^0 = \{w :  (0 (K{-}1))^\infty \less w_{[i, \infty)} \lesseq K^\infty \text{ for all } i \geq 1 \}.$$
Thus, $\mathcal{W}_{-K}^0$ consists of words over the alphabet $\{0,1,\dots,K\}$ with some restrictions. 
The first inequality implies that words cannot contain the string $0K$ or end in $(0(K{-}1))^\infty$.   
The second inequality implies that if $w\in\mathcal{W}_{-K}^0$ is such that $w_i = K$ for some $i$, then we must have 
$w = w_1 w_2 \dots w_{i -1} K^\infty$, and  $w_{i - 1} \neq 0$ in order to avoid $0K$.  It follows that
\begin{multline*}
\mathcal{W}_{-K}^0 = \mathcal{W}_K \cup \{ w: w = w_1 w_2 \dots w_i K^\infty \text{ for some } i \geq 0, \text{ with } w_i \neq 0\text{ and }w_j\leq K-1\text{ for }1\le j\le i\}\\ 
 \setminus \{  w: w=w_1 w_2 \dots w_i (0 (K{-}1))^\infty \text{ for some } i \geq 0\}.
 \end{multline*}
On the other hand, recall that in Equation~\eqref{eq:WN0} we defined
$$\mathcal{W}_{K}^0 = \mathcal{W}_{K} \setminus \{  w: w=w_1 w_2 \dots w_i (0 (K{-}1))^\infty  \text{ and } w_i \neq K{-}1, \text{ for some } i \geq 1 \}. $$
Let $w \in \mathcal{W}_{K}^0 \setminus \mathcal{W}_{-K}^0$.  Then $w = w_1 w_2 \dots w_{j} ((K{-}1)0)^\infty$ for some $j \geq 0$ and $w_j \neq 0$.  If $j < n-2$, then the pattern of length $n$ for $w$ is undefined.  If $j \geq n-2$, then $w' = w_1 w_2 \dots w_j ((K{-}1)0)^n 0^\infty \in \mathcal{W}_{-K}^0$ induces the same pattern of length $n$ as $w$.  Hence, $\Al(\Sigma_{-K}|_{\mathcal{W}_{K}^0}) \subseteq \Al(\Sigma_{-K}|_{\mathcal{W}_{-K}^0})$. 

Let now $w \in \mathcal{W}_{-K}^0 \setminus \mathcal{W}_{K}^0$.  Then $w = w_1 w_2 \dots w_{i} K^\infty$ for some $i\ge0$, with $w_i \neq 0$.  If $i < n-1$, then the pattern for $w$ is not defined.  If $i \geq n-1$, then the word $w' = w_1 w_2 \dots w_i ((K{-}1)0)^\infty \in \mathcal{W}_{K}^0$ induces the same pattern of length $n$ as $w$.   Hence, $\Al(\Sigma_{-K}|_{\mathcal{W}_{K}^0}) \supseteq \Al(\Sigma_{-K}|_{\mathcal{W}_{-K}^0})$.  
\end{proof}

\begin{defn}  A $-\beta$-representation of $x \in [0, 1]$ is any word $w \in \WN$ that satisfies $f_w(\beta) = x$ and $f_{w_{[k, \infty)}}(\beta) \in [0, 1]$ for all $k \geq 1$.  \end{defn}

By definition, $\Wb$ is the set of all $-\beta$-representations of numbers in $[0, 1]$. We will see that even though the word $\wmax$ is always a $-\beta$-representation of $1$, it is not always a $-\beta$-expansion. 
The following lemma characterizes which $-\beta$-representations are in fact $-\beta$-expansions.

\begin{lem} \label{isexpansion} If $w \in \Wb$ is such that $f_{w_{[k, \infty)}}(\beta) \in (0, 1]$ for all $k \geq 1$, then $w \in \Wb^0$.  \end{lem}

\begin{proof} Let $v\in\Wb^0$ be the $-\beta$-expansion of the point $f_{w}(\beta) \in (0, 1]$. We will show that $w=v$. Suppose not,  and let $i$ be the smallest index such that $w_i \neq v_i$.  Then 
$$0 = f_{w}(\beta) - f_{v}(\beta) 
= \frac{1}{(-\beta)^i} \left( (v_i - w_i) + ( f_{w_{[i+1, \infty)}}(\beta) - f_{v_{[i+1, \infty)}}(\beta)) \right).$$
Since $f_{w_{[i+1, \infty)}}(\beta) \in (0, 1]$ by assumption, and $f_{v_{[i + 1, \infty)}}(\beta) \in (0, 1]$ 
by Lemma \ref{itos} and the fact that $v_{[i+1, \infty)} \in \Wb^0$, we have that $| f_{w_{[i+1, \infty)}}(\beta) - f_{v_{[i+1, \infty)}}(\beta)  | < 1$. But $|v_i - w_i| \geq 1$, and so the above equality is impossible.  \end{proof}

It follows from Lemma~\ref{isexpansion} that $$\Wb^0 = \{ w\in\WN : 0<  f_{w_{[k, \infty)}}(\beta) \leq 1\text{ for all } k \geq 1 \}.$$ 
Lemma~\ref{orderinWb0} can be extended to the set $\Wb$ as follows.

\begin{lem} \label{monotone} Let $v, w \in \Wb$. If $f_{v}(\beta) < f_w(\beta)$, then $v \less w$. Equivalently, if $w \lesseq v$, then $f_w(\beta) \leq f_v(\beta)$. 
\end{lem}

\begin{proof} Suppose that $f_{v}(\beta) < f_w(\beta)$. Let $i$ be the smallest index such that $w_i \neq v_i$.  Then 
$$0 < f_w(\beta) - f_v(\beta) = \frac{1}{\beta^i} \left( (-1)^i(v_i - w_i) + (-1)^i( f_{w_{[i+1, \infty)}}(\beta) - f_{v_{[i+1, \infty)}}(\beta)) \right).$$
Since $v,w \in \Wb$, we have $| f_{w_{[i+1, \infty)}}(\beta) - f_{v_{[i+1, \infty)}}(\beta) | \leq 1$.  Therefore, $(-1)^i(v_i - w_i) > 0$ and we conclude that $v \less w$.
\end{proof}

The following lemma will be used to give an equivalent description of $\Wb$ in Lemma~\ref{lem:Wb}. Its proof is along the lines of a similar statement found in \cite{ItoS} for $\Wb^0$. 

\begin{lem} \label{lem:parrystyle} Let $w \in \WN$ be a word such that $w_{[k, \infty)} \lesseq \wmax$ for all $k \geq 1$.
Then $f_{w_{[k, \infty)}}(\beta) \leq f_{{\wmax}}(\beta) = 1$ for all $k \geq 1$.
\end{lem}

\begin{proof} We extend the definition of $f_w(\beta)$ to finite words $a_1\dots a_m$ by letting \\ $f_{v_1\dots v_m}(\beta) = - \sum_{j = 1}^m \frac{v_j + 1}{(-\beta)^j}$. We also extend $\less$ by defining $v_1\dots v_m \less w_1\dots w_m$ if there exists some $i \leq m$ such that $v_j = w_j$ for all $j < i $ and $(-1)^i(v_i - w_i) > 0$.    

We will show by induction on $r$ that, for every $i, j \geq 1$,
\begin{enumerate}[a)]
\item $w_{[i, i + r]} \gesseq {\wmax}_{[j, j + r]}$ implies $f_{w_{[i, i + r]}}(\beta) \geq f_{{\wmax}_{[j, \infty)}}(\beta) - \frac{1}{\beta^{r+1}}$, and
\item $w_{[i, i + r]} \lesseq {\wmax}_{[j , j + r]}$ implies $f_{w_{[i, i + r]}}(\beta) \leq f_{{\wmax}_{[j, \infty)}}(\beta) + \frac{1}{\beta^{r + 1}}$.  
\end{enumerate}

Consider first the case $r = 0$.  If $w_i \gesseq {\wmax}_j$, then 
\begin{eqnarray*}
f_{w_{i}}(\beta) &=& - \frac{w_i + 1}{-\beta} \geq - \frac{{\wmax}_j+1}{-\beta} 
 =  f_{{\wmax}_{[j, \infty)}}(\beta) - \frac{1}{\beta} f_{{\wmax}_{[j + 1, \infty)}}(\beta) 
\geq f_{{\wmax}_{[j, \infty)}}(\beta) - \frac{1}{\beta}.
\end{eqnarray*}
Similarly, if $w_i \lesseq {\wmax}_j$, then 
\begin{eqnarray*} 
f_{w_i}(\beta) &=& - \frac{w_i+1}{-\beta} \leq - \frac{{\wmax}_{j}+1}{-\beta} 
= f_{{\wmax}_{[j, \infty)}}(\beta) - \frac{1}{\beta} f_{{\wmax}_{[j + 1, \infty)}}(\beta) \leq f_{{\wmax}_{[j, \infty)}}(\beta) + \frac{1}{\beta}. \end{eqnarray*}
Therefore, both a) and b) hold when $r = 0$.

Now fix $k\ge1$, and assume that a) and b) hold for all $i, j \geq 1$ whenever $r < k$.  
To prove statement b) for $r=k$, suppose that $w_{[i, i + k]} \lesseq {\wmax}_{[j, j + k]}$. Then either $w_i = w_j$ and $w_{[i + 1, i + k]} \gesseq {\wmax}_{[j + 1, j + k]}$, or else $w_i < {\wmax}_j$.  In the first case,  we have
$$
 f_{{\wmax}_{[j, \infty)}}(\beta)-  f_{w_{[i, i + k]}}(\beta) = \frac{1}{-\beta}(f_{{\wmax}_{[j + 1, \infty)}}(\beta)- f_{w_{[i + 1, i + k]}}(\beta)) \geq \frac{1}{-\beta}\cdot\frac{1}{\beta^k} = - \frac{1}{\beta^{k+1}}, 
$$
where, to obtain the inequality, we applied the induction hypothesis for statement a).  On the other hand, if $w_i < {\wmax}_j$, we have
\begin{multline*} 
 f_{{\wmax}_{[j, \infty)}}(\beta) -f_{w_{[i, i + k]}}(\beta)  =  - \frac{ {\wmax}_j - w_i}{-\beta} + \frac{1}{-\beta}( f_{{\wmax}_{[j + 1, \infty)}}(\beta) -f_{w_{[i + 1, i + k]}}(\beta)  ) \\
\geq \frac{1}{\beta}(1 + (f_{w_{[i + 1, i + k]}}(\beta) -   f_{{\wmax}_{[j + 1, \infty)}}(\beta))
 \geq \frac{1}{\beta}\left(1 + \left(- \frac{1}{\beta^k} - 1\right)\right) =  - \frac{1}{\beta^{k + 1}},
\end{multline*}
where in the last inequality we used that $f_{w_{[i + 1, i + k]}}(\beta) \geq -\frac{1}{\beta^k}$. To see why this holds, note that $w_{[i+2, i + k]} \lesseq {\wmax}_{[1, k-1]}$ by assumption, and $f_{w_{[i + 2, i + k]}}(\beta) \leq f_{\wmax}(\beta) + \frac{1}{\beta^{k-1}}$ by the induction hypothesis, so that
$$f_{w_{[i + 1, i + k]}}(\beta) = \frac{w_{i + 1} + 1}{\beta} - \frac{1}{\beta}f_{w_{[i + 2, i + k]}}(\beta) \geq \frac{1}{\beta} - \frac{1}{\beta} \left( f_{\wmax}(\beta) + \frac{1}{\beta^{k-1}} \right) = -\frac{1}{\beta^{k}}.$$
The proof of statement a) for $r=k$ follows similarly, and thus this completes the inductive proof of a) and b).

Taking the limit of b) as $r \rightarrow \infty$, we find 
$$w_{[i, \infty)}  \lesseq {\wmax}_{[j, \infty)}  \text{ implies }f_{w_{[i, \infty)}}(\beta) \leq  f_{{\wmax}_{[j, \infty)}}(\beta)  .$$
In particular, setting $i=k$ and $j=1$, the assumption that $w_{[k, \infty)} \lesseq \wmax$ implies that $f_{w_{[k, \infty)}}(\beta) \leq f_{\wmax}(\beta) = 1$.  
\end{proof}

Next we give an equivalent description of $\Wb$.

\begin{lem}\label{lem:Wb} $\wmax$ is the largest $-\beta$-representation of $1$ with respect to $\less$, and
\begin{equation}\label{eq:Wb2}\Wb = \{ w\in\WN : w_{[k, \infty)} \lesseq \wmax \text{ for all } k \geq 1 \}.
\end{equation}
\end{lem}

\begin{proof} Since $\betaone\in\Wb$ and $\wmax$ is the largest word in $\Wb$ by definition, we have that $\betaone\lesseq\wmax$.
By Lemma \ref{monotone}, it follows that $1 = f_{\betaone}(\beta) \leq f_{\wmax}(\beta)\le 1$, and so $f_{\wmax}(\beta) = 1$. Thus, $\wmax$ is a $-\beta$-representation of 1, hence the largest. 

Next let us prove Equation~\eqref{eq:Wb2}. Let $w\in\Wb$. By Definition~\ref{def:Wb}, $w_{[k, \infty)} \lesseq \wmax$ for all $k \geq 1$, since $\wmax$ is the largest word in $\Wb$. This proves the forward inclusion. 
Conversely, let $w \in \WN$ be such that $w_{[k, \infty)} \lesseq \wmax$ for all $k \geq 1$.  
By Lemma \ref{lem:parrystyle}, $f_{w_{[k, \infty)}}(\beta) \leq 1$ for all $k \geq 1$.  To show $f_{w_{[k, \infty)}}(\beta)\ge 0$, suppose for contradiction that $f_{w_{[k, \infty)}}(\beta) =- \frac{w_k + 1}{-\beta}  + \frac{1}{-\beta}f_{w_{[k + 1, \infty)}}(\beta)< 0$ for some $k \geq 1$. Since $w_k\ge0$, this would imply that $f_{w_{[k +1, \infty)}}(\beta) > 1$, a contradiction.  Thus, $f_{w_{[k, \infty)}}(\beta) \in [0, 1]$ for all $k \geq 1$, and so $w\in\Wb$. 
\end{proof}

A consequence of Lemma~\ref{lem:Wb} is that 
\begin{equation}\label{eq:wmaxsubwords}\wmax\gesseq{\wmax}_{[k,\infty)} \text{ for all }k,\end{equation} that is, $\wmax$ is greater than or equal to all of its shifts.

Since $\Wb$ is closed under shifts, and $\wmin$ is the smallest word in $\Wb$ by definition, an equivalent description of $\Wb$ is 
$$\Wb = \{ w\in\WN : \wmin \lesseq w_{[k, \infty)} \lesseq \wmax \text{ for all } k \geq 1 \}.$$
Note also that since $\wmin = 0 \wmax$ and $f_{\wmax}(\beta)=1$, we have $f_{\wmin}(\beta) = 0$ by definition of $f$. Hence, $\wmin$ is the smallest $-\beta$-representation of $0$ with respect to~$\less$.

If $\betaone$ is not eventually periodic, then $\betaone$ is the unique $-\beta$-representation of $1$, and $\betaone=\wmax$. If $\betaone = (a_1 a_2 \dots a_{2r+1})^\infty$ is periodic of odd length $2r+1$, another $-\beta$-representation of $1$ is $(a_1 a_2 \dots (a_{2r+1} -1)0)^\infty$.   In this case, $\betaone=\wmax$ is the largest $-\beta$-representation of $1$.  If $\betaone = (a_1 a_2 \dots a_{2r})^\infty$ is periodic of even length $2r$, then $\wmax = (a_1 a_2 \dots (a_{2r} - 1)0)^\infty\gess\betaone$.  Similar observations were first made in \cite{Steiner}. We will use the following result of Steiner. 

\begin{thm}[\cite{Steiner}] \label{Steiner} Let $\beta, \beta' > 1$. Similarly to the definition of $\betaone$, let $\mathfrak{a}_{\beta'}$ be the $-\beta'$-expansion of 1.  Then $\beta < \beta'$ if and only if $\betaone \less \mathfrak{a}_{\beta'}$.
 \end{thm}

Let $u=100111001001001110011\dots$ be the sequence obtained by starting with the word $1$ and repeatedly applying the morphism $1\mapsto100$, $0\mapsto1$.  It is shown in \cite{LS}  that the word $u$  is the limit of the words $\betaone$ as $\beta$ approaches $1$ from the right, and that
\begin{equation}\label{eq:u-ab} u\less \betaone \text{ for all }\beta > 1.
\end{equation}

\begin{thm}[\cite{Steiner}] \label{SteinerSoln} Let $w\in\WN$ be such that $w \gesseq w_{[k, \infty)} $ for all $k \geq 1$ and $w \gess u $.  Then there exists a unique $\beta > 1$ such that $w$ is a $-\beta$-representation of $1$.
\end{thm}

\begin{lem} \label{alphabetsubset} If $1 < \beta < \beta'$, then $\Wb \subsetneq \mathcal{W}_{-\beta'}$.  \end{lem} 

\begin{proof}
By Theorem \ref{Steiner} and Equation \eqref{eq:u-ab}, $u\less \betaone \less \mathfrak{a}_{\beta'}$.  Let us first show that $\mathfrak{a}_{\beta'} \notin \Wb$. 
If we were to have $\mathfrak{a}_{\beta'} \in \Wb$, then Lemma \ref{monotone} would imply that $f_{\mathfrak{a}_{\beta'}}(\beta) = 1$. Since $\mathfrak{a}_{\beta'}$ is a $-\beta'$-expansion of $1$, we have ${\mathfrak{a}_{\beta'}}_{[k, \infty)} \lesseq \mathfrak{a}_{\beta'}$ for all $k \geq 1$.  If follows from Lemma \ref{orderinWb0} that $f_{{\mathfrak{a}_{\beta'}}_{[k, \infty)}}(\beta) \in [0, 1]$ for all $k \geq 1$.  Hence, $\mathfrak{a}_{\beta'}$ is both a $-\beta$-representation of $1$ and a $-\beta'$-representation of $1$, contradicting Theorem \ref{SteinerSoln}.

The fact that $\mathfrak{a}_{\beta'} \notin \Wb$ implies, by Lemma~\ref{lem:Wb}, that there is a $k \geq 1$ such that ${\mathfrak{a}_{\beta'}}_{[k, \infty)} \gess \wmax$.  Since $\mathfrak{a}_{\beta'}$ is a $-\beta'$-expansion of $1$, $\mathfrak{a}_{\beta'} \gesseq {\mathfrak{a}_{\beta'}}_{[k, \infty)}$ for all $k \geq 1$.  We conclude that $\wmax \less \mathfrak{a}_{\beta'}$. It follows that, if $v\in\Wb$, then  $v_{[k, \infty)} \less \mathfrak{a}_{\beta'}$ for all $k \geq 1$, and we conclude that $v \in \mathcal{W}_{-\beta'}$.  Moreover, containment is strict because $\mathfrak{a}_{\beta'} \notin \Wb$.  
\end{proof}

\begin{lem} \label{lem:unique_largest}
In the situation of Theorem~\ref{SteinerSoln}, the unique $\beta>1$ is also the largest real solution of $f_w(x)=1$.
\end{lem}

\begin{proof}
Suppose for contradiction that there exists $\gamma>\beta$ such that $f_{w}(\gamma) = 1$. By Lemma \ref{alphabetsubset}, $w \in \Wb \subseteq \mathcal{W}_{-\gamma}$,
and so $f_{w_{[k,\infty)}}(\gamma)\in[0,1]$ for all $k\ge1$. Since $f_{w}(\gamma) = 1$, the word $w$ is a $-\gamma$-representation of 1, contradicting the uniqueness in Theorem~\ref{SteinerSoln}.
\end{proof}

\begin{defn} \label{def: betaw} For a given word $w \in \WN$, let 
$$\bar\beta(w) = \inf \{ \beta>1 : w \in \Wb \}.$$
\end{defn}

\begin{defn} \label{barb} Let $w \in \WN$. If there is an index $l$ such that $w_{ [k, \infty)} \lesseq w_{[l, \infty)}$ for all $k \geq 1$ and $w_{[l, \infty)} \gess u$, let $\b(w)$ be the largest real solution to $f_{w_{[l, \infty)}}(x) = 1$ (equivalently, by Lemma~\ref{lem:unique_largest},
$\b(w)$ is the unique $\beta>1$ such that such that $w_{[l, \infty)}$ is a $-\beta$-representation of $1$).  If $w_{[k, \infty)} \lesseq u$ for all $k \geq 1$, define $\b(w) = 1$.
 \end{defn}

\begin{lem}\label{betcor} If there is an index $l$ such that $w_{[k, \infty)} \lesseq w_{[l, \infty)}$ for all $k \geq 1$, then 
$$\bar\beta(w) = \b(w).$$
Additionally, if $\bar \beta(w) > 1$, then $w\in\mathcal{W}_{- \bar\beta(w)}$; if $\bar \beta(w) = 1$, then $w \in \Wb$ for all $\beta > 1$.
\end{lem}

\begin{proof} We consider two cases, depending on whether $w_{[l, \infty)} \gess u$ or $w_{[l, \infty)} \lesseq u$ (which implies that 
$w_{[k, \infty)} \lesseq u$ for all $k \geq 1$). By Definition~\ref{barb}, these cases correspond to $\b(w) > 1$ and to $\b(w)=1$, respectively.

For the case $w_{[l, \infty)} \gess u$, let $\beta = \b(w)$. We will show that $w\in\Wb$ and that $w\notin\mathcal{W}_{-\gamma}$ for $\gamma<\beta$,  where it will follow that $\bar\beta(w)=\beta$.  By Definition \ref{barb}, $w_{[l, \infty)}$ is a $-\beta$-representation of 1.  Since $w_{[k, \infty)} \lesseq w_{[l, \infty)}$ for all $k \geq 1$, we have $w \in \Wb$.  Now let $\gamma < \beta$ and suppose for contradiction that $w \in \mathcal{W}_{-\gamma}$.  Then $f_{w_{[k, \infty)}}(\gamma) \leq 1$ for all $k \geq 1$ and we first claim this inequality is strict.  If $f_{w_{[l, \infty)}}(\gamma) = 1$, the word $w_{[l, \infty)}$ would be a $-\gamma$-representation of $1$.  However, this is a contradiction to Theorem \ref{SteinerSoln} since $w_{[l, \infty)}$ is already a $-\beta$-representation of $1$.  Therefore, $f_{w_{[l, \infty)}}(\gamma) < 1 = f_{\mathfrak{a}_{\gamma}}(\gamma)$.  By Lemma \ref{monotone} and Theorem \ref{Steiner}, we must have $w_{[l, \infty)} \less \mathfrak{a}_{\gamma} \less \mathfrak{a}_{\beta}$.  By Lemma \ref{monotone} and the fact that $f_{w_{[l, \infty)}}(\beta)=1$, we conclude that $f_{\mathfrak{a}_{\gamma}}(\beta) = 1$.  However, this is impossible by Theorem \ref{SteinerSoln} because $\mathfrak{a}_{\gamma}$ is already a $-\gamma$-representation of $1$.  We  may now conclude that $w \notin \mathcal{W}_{-\gamma}$ and $\bar{\beta}(w) = \b(w)$ in the case that $\b(w) > 1$.  

Now consider the case $w_{[l, \infty)} \lesseq u$. Take any $\beta > 1$.  By Equation \eqref{eq:u-ab}, $u \less \betaone$.  
Since $w_{[k, \infty)} \lesseq u\less \betaone$ for all $k \geq 1$, we have that $w \in \Wb$. Since this holds for all $\beta>1$, it follows that $\bar{\beta}(w)=1$. We conclude that $\bar{\beta}(w) = \b(w)$ in all cases.   
\end{proof}

\begin{lem} If $1 < \beta \leq \beta'$, then 
$$\Al(\Sb) \subseteq \Al(\Sigma_{-\beta'})$$
\end{lem}

\begin{proof} This follows from  Lemma \ref{alphabetsubset} and the fact that for $w \in \Wb\subseteq\mathcal{W}_{-\beta'}$, we have $$\Pat(w, \Sigma_{-\beta}, n) =\Pat(w, \Sigma_{-\beta'}, n).$$  \end{proof}

\begin{defn}\label{def:Bpi} For any permutation $\pi$, let 
$$\B(\pi) = \inf \{ \beta : \pi \in \Al(\Sigma_{-\beta}) \}.$$ 
Equivalently, 
$$\B(\pi) = \inf \{ \bar\beta(w) : \Pat(w, \Sigma_{-\beta}, n) = \pi \}.$$
\end{defn}  

We call $\B(\pi)$ the {\em negative shift-complexity} of $\pi$.
Alternatively, $\B(\pi)$ is the supremum of the set of values $\beta$ such that $\pi$ is a forbidden pattern of $\Sigma_{-\beta}$. Thinking of  $\Sigma_{-\beta}$ as a family of maps parametrized by $\beta$, the numbers of the form $\B(\pi)$ are the values of $\beta$ where we obtain additional patterns as we increase $\beta$. In the rest of this section, we show that these values are the same for the $-\beta$-transformation $\Tb$.

\begin{lem} \label{inTb} If $\pi \in \Al(\Sigma_{-\beta})$ and $\gamma > \beta$,  then $\pi \in \Al( T_{-\gamma})$. 
\end{lem}

\begin{proof} Suppose that $\pi \in \Al(\Sigma_{-\beta})$ and $\gamma>\beta$. Let $w \in \Wb \subseteq \mathcal{W}_{-\gamma}$ be such that \\ $\Pat(w, \Sigma_{-\beta}, n) = \pi$.  We claim that $f_{w_{[k, \infty)}}(\gamma) \in (0, 1)$ for all $k \geq 1$.  By Theorem \ref{SteinerSoln}, since $\wmax$ is a $-\beta$-representation of $1$, $\wmax$ cannot also be a $-\gamma$-representation of $1$.  Additionally, since $w \in \Wb$, we also have $w_{[k, \infty)} \lesseq \wmax$ for all $k \geq 1$.   Lemma \ref{monotone} implies that $f_{w_{[k, \infty)}}(\gamma)\le f_{\wmax}(\gamma) < 1$ for all $k \geq 1$. Moreover, for all $k \geq 1$, we have $f_{w_{[k, \infty)}}(\gamma) = -\frac{w_k + 1}{-\gamma} + \frac{1}{-\gamma} f_{w_{[k + 1, \infty)}}(\gamma) > 0$,  because $w_k \geq 0$ and $f_{w_{[k + 1, \infty)}}(\gamma) < 1$.  We conclude that $f_{w_{[k, \infty)}}(\gamma) \in (0, 1)$ for all $k \geq 1$.  

By Lemma \ref{isexpansion}, $w$ is the $-\gamma$-expansion of the point $f_{w}(\gamma) \in (0, 1)$, and $w \in \mathcal{W}_{- \gamma}^0$.  By Lemma~\ref{itos}, $f$ gives an order isomorphism between the map $T_{-\gamma}$ on $((0, 1], <)$ and the map $\Sigma_{-\gamma}$ on $(\mathcal{W}_{-\gamma}^0, \less)$.  Hence,  $\Pat(f_w(\gamma), T_{-\gamma}, n) = \Pat(w, \Sigma_{-\gamma}, n) = \Pat(w, \Sigma_{-\beta}, n) = \pi$,
and so $\pi \in \Al(T_{-\gamma})$.
  \end{proof}

The next result shows that the definition of $\B(\pi)$ is not affected if we consider the map $\Tb$ instead of $\Sb$.

\begin{prop} \label{domain}
$$\B(\pi) = \inf \{ \beta : \pi \in \Al(\Tb) \}.$$ 
\end{prop}

\begin{proof} By Lemma~\ref{itos}, $\Sb$ on $(\Wb^0, \less)$ and $\Tb$ on $((0, 1], < )$ are order-isomorphic. Since we consider the domain of $\Sb$ to be $\Wb\supseteq\Wb^0$, we have that $\Al(\Sb)\supseteq\Al(\Tb)$, and so $\B(\pi)=\inf \{ \beta : \pi \in \Al(\Sb) \}\le\inf \{ \beta : \pi \in \Al(\Tb) \}$.
To prove the inequality $\B(\pi)\ge\inf \{ \beta : \pi \in \Al(\Tb) \}$, we show that if $\beta > \B(\pi)$, then $\pi \in \Al(\Tb)$. 
To see this, let $\gamma = \frac{1}{2}(\beta + \B(\pi))$.  Since $\gamma>\B(\pi)$, we have $\pi \in \Al(\Sigma_{-\gamma})$.  By Lemma \ref{inTb}, $\beta > \gamma$ implies that $\pi \in \Al(\Tb)$.
\end{proof}

\section{Building words}\label{sec:words}

In the remaining two sections we give a method to compute $\B(\pi)$ for any given permutation $\pi$. The idea of this section is to construct a word $w$ such that $\B(\pi) = \bar{\beta}(w)$. This word will have an index $l$ such that $w_{[l, \infty)} \gesseq w_{[k, \infty)}$ for all $k \geq 1$, and so we can apply Lemma~\ref{betcor} to deduce that $\bar{\beta}(w) = \b(w)$.  In Section~\ref{sec:B}, we will express this quantity as the largest real solution to a polynomial. 
In the rest of the paper, we will use the term {\em subword} of $w$ to specifically mean a word of the form $w_{[i, \infty)}$ for some $i \geq 1$ (sometimes this is also called a shift of $w$).

The construction depends on features of $\pi$ such as the parity of $n-\ell$ (where $\ell$ is the index such that $\pi_\ell=n$) and whether $\pi$ is regular, cornered or collapsed. In nearly every case, we define a collection of words $w^{(m)}$ such that $w^{(m)}$ induces $\pi$ for $m \geq n-1$, and given any other $v \in \WN$ inducing $\pi$, there is an $m$ large enough so that $\bar{\beta}(w^{(m)}) \leq \bar{\beta}(v)$.  
This inequality will follow from Lemma \ref{betab} using that $w^{(m)}_{[l, \infty)} \less v_{[l, \infty)}$ for sufficiently large $m$.  The sequence of words $w^{(m)}$ will be constructed so that, as $m \rightarrow \infty$, it approaches a fixed word $w$ with maximal subword $w_{[\ell, \infty)}$.  Moreover, $w$ satisfies $\B(\pi) = \bar{\beta}(w) = \b(w)$, where the last equality follows from Lemma \ref{betcor}. 

In the rest of this section, fix $N=\N(\pi)$, and let $1<\beta \leq N$. Let $\zeta$ be a prefix defined by a valid $-N$-segmentation of $\hat{\pi}$ (as in Definition~\ref{segmentation}), which exists by Lemma \ref{geq}. Recall that $\zeta$ is uniquely determined if $\pi$ is regular, by Lemma~\ref{numberseg}. Define $x$, $y$, $p$ and $q$ as in Section~\ref{sec:-N}, and let $\ell$ be the index such that $\pi_{\ell} = n$. Let $\pi_h$ be the maximum of $\pi_x, \pi_{x + 1}, \pi_{x + 2}, \dots \pi_n$, and notice that if $\ell \geq x$, then $h = \ell$.  

When $n -\ell$ is odd (which implies that $\pi_n \neq n$), we define $s^{(m)} = \zeta p^{2m} \wend$ where 
\begin{equation}\label{eq:smdef} 
\wend = \begin{cases}
z_{[x, h-1]} \wmin & \text{if both }  h-x  \text{ and } |p| \text{ are even} \hfill (\ref{eq:smdef}.1)\\
&  \text{(unless the condition before Equation \eqref{eq:wend} holds),} \\
p z_{[x, h-1]} \wmax & \text{if  } h-x  \text{ is even  and }  |p| \text{ is odd,} \hfill (\ref{eq:smdef}.2)\\
z_{[x, h-1]} \wmax & \text{if } h - x \text{ is odd.} \hfill (\ref{eq:smdef}.3)
\end{cases} 
\end{equation}

In the special case that $h-x$ and $|p|$ are even and there is an index $x \le j < h$ such that $h - j$ is odd and $z_{[j, h-1]} = z_{[\ell, \ell + h - j -1]}$, we define
\begin{equation} \label{eq:wend}  \wend = z_{[x, h-1]}(z_{[j, h-1]})^\infty, \end{equation}
where $j$ is chosen so that $\pi_j > \pi_{j'}$ for any other indices $j'$ with this property. 

When $n - \ell$ is even and $\pi_n \neq 1$, we define $t^{(m)} = \zeta q^{2m} \wend$ where 
\begin{equation}\label{eq:tmdef} 
\wend= \begin{cases}
z_{[y, h-1]} \wmin & \text{ if } h-y \text{ is odd and } |q| \text{ is even}  \hfill (\ref{eq:tmdef}.1)\\ 
& \text{ (unless the condition before Equation \eqref{eq:wendq} holds),} \\
qz_{[y, h-1]} \wmax & \text{ if both } h-y \text{   and } |q| \text{ are odd,} \hfill (\ref{eq:tmdef}.2)\\
z_{[y, h-1]} \wmax & \text{ if } h - y \text{ is even.} \hfill (\ref{eq:tmdef}.3)
\end{cases}
\end{equation}
In the special case that $h-x$ is odd, $|q|$ is even, and there is an index $y \leq j < h$ such that $h - j$ is odd and $z_{[j, h-1]} = z_{[\ell, \ell + h - j -1]}$, we define \begin{equation} \label{eq:wendq} \wend = z_{[y, h-1]}(z_{[j, h-1]})^\infty,  \end{equation} where we choose the index $j$ such that $\pi_j > \pi_{j'}$ for all other indices $j'$ with this property.

Note that for permutations $\pi$ satisfying $n-\ell$ is even and $\pi_n=1$, neither $s^{(m)}$ nor $t^{(m)}$ is defined. We will deal with this case separately in Proposition \ref{prop:regularnone}.

The rest of the section is dedicated to proving the propositions listed in Table \ref{tab:cases} which, for a given permutation $\pi$, describe a word $w$ such that $\B(\pi) = \b(w)$. In most cases, the word $w$ arises by taking the limit as $m \rightarrow \infty$ of either of the words $s^{(m)}$ or $t^{(m)}$.  We will show in Lemmas \ref{lem:pwless}, \ref{betainduces} and \ref{lem:inWb} that for $\beta$ large enough, $\Pat(s^{(m)}, \Sigma_{-}, n) = \pi$ and $\Pat(t^{(m)}, \Sigma_{-}, n) = \pi$ in the cases when $s^{(m)}$ and $t^{(m)}$ are defined. Moreover, for such a $\beta$, we will show that there is a sufficiently large $m$ such that $s^{(m)}, t^{(m)} \in \Wb$.  We begin with Lemmas \ref{dinfinity}, \ref{largeenough} and \ref{betab} to establish a few properties of words in $\Wb$ that will be important throughout the section.

\begin{example} To illustrate the need for the special case in Equation~\eqref{eq:wend}, consider $\pi = 81735642$, a regular permutation such that $n-\ell=8-1$ is odd, $h-x=6-4$ even and $|p|=4$ even.  Then $\hat{\pi} = 7{\star}526431$, $\asc(\hat{\pi}) = 1$ and $\N(\pi) = 2$.  A valid $-2$-segmentation of $\hat{\pi}$ is given by $(e_0, e_1, e_2) = (0, 4, 8)$,
 defining the prefix $\zeta = 1010110$.  

Suppose that we defined $\wend = z_{[x, h-1]} \wmin = 01 \wmin$, following Equation \eqref{eq:smdef}. In order to have $\wend \in \Wb$, each subword of $\wend$ would be in $\Wb$, and so $1 \wmin \lesseq \wmax$.  This inequality becomes $1 0 \wmax \lesseq \wmax$ and, by Lemma \ref{dinfinity} below, this would imply that $\wmax \gesseq (10)^\infty$.  Thus, for $\beta < 2$, we have $\wend \notin \Wb$ and $s^{(m)} \notin \Wb$.  

However, a small adjustment to $\wend$, as described in Equation \eqref{eq:wend}, yields $\wend \in \Wb$ for all $\beta > \b(\zeta p^\infty) \approx 1.96$.   In this case, the index $j = 5$ is such that $h - j$ is odd and $z_{[j, h - 1]} = z_{[\ell, \ell + h - j - 1]} = 1$, and so we take $\wend =  z_{[x, h-1]} (z_{[j, h-1]})^\infty = 01(1)^\infty = 01^\infty$.  We will verify in this section that such a choice of $\wend$ satisfies the conditions needed for $s^{(m)}$ to induce $\pi$, and also that $s^{(m)} \in \Wb$ for all $\beta > \b(\zeta p^\infty)$.  
\end{example}

\begin{lem} \label{dinfinity} Let $d$ be a finite word, and $v$ an infinite word. The following are equivalent:
\begin{center}
(a) $v \gess d^\infty$; \qquad (b) $v \gess d^m v$ for all $m \geq 1$; \qquad (c)  $v \gess d v$.
\end{center}
Likewise, if we replace $\gess$ with $\less$ in (a),(b),(c), the resulting three statements are equivalent to each other. \end{lem} 

\begin{proof} We will prove (a) implies (b) and that (c) implies (a). The fact that (b) implies (c) is trivial. The proof of the corresponding statements for $\less$ is analogous.

To show that (a) implies (b), first suppose that $|d|$ is odd.  For all $i\ge1$, (a) implies that $v\gess d^{\infty} \gess d^{2i-1}v$, proving (b) when $m$ is odd.  Now suppose that we had $v\lesseq d^{2i} v$ for some $i$.  Then $d^{2i-1}v \less v \lesseq d^{2i} v$, which forces $v = d^\infty$, causing a contradiction. This proves (b) for even $m$.

Now consider the case when $|d|$ is even. Suppose for contradiction that $v\lesseq d^m v $ for some $m \geq 1$.  Then 
$v\lesseq d^m v\lesseq d^{2m} v \lesseq \dots \lesseq d^{km}v$ for all $k \geq 1$, contradicting (a). 

To prove that (c) implies (a), first consider the case when $|d|$ is odd. Suppose we had $v \lesseq d^2 v$.  Then
$dv \less v \lesseq d^2 v$,
which forces $v = d^\infty$, causing a contradiction.  Therefore, $v \gess d^2 v$, and we obtain $v \gess d^2 v \gess d^4 v \gess \ldots \gess d^{2k} v$ for all $k\ge 1$.  In the case when $|d|$ is even, $v \gess d v$ implies that $v \gess d v \gess d^2 v \gess \ldots \gess d^k v$ for all $k \geq 1$.  We conclude that (a) holds in all cases. 
\end{proof}

\begin{lem} \label{largeenough} Let $w \in \WN$ be a word for which there is an index $l$ such that $w_{[l, \infty)} \gesseq w_{[k, \infty)}$ for all $k \geq 1$. If $\beta > \b(w)$, then $\wmin \less w_{[k, \infty)} \less \wmax$ for all $k \geq 1$. \end{lem} 

\begin{proof}  Fix $\beta > \b(w) \geq 1$. Then Lemmas~\ref{betcor} and~\ref{alphabetsubset} imply that $w\in\Wb$, and so $\wmin \lesseq w_{[k, \infty)} \lesseq \wmax$ for all $k \geq 1$.  It remains to show that both inequalities are strict.  

First consider the case $w_{[l, \infty)} \gess u$. If we had $w_{[k, \infty)} = \wmax$ for some $k$, then $w_{[l, \infty)} = \wmax$, and so $f_{w_{[l, \infty)}}(\beta) = 1$.  Since, by Definition \ref{barb}, $\b(w)$ is the largest real solution to $f_{w_{[l, \infty)}}(x) = 1$, this implies $\b(w) \geq \beta$, a contradiction.  Hence, there is no $k \geq 1$ such that $w_{[k, \infty)} = \wmax$.  Since $\wmin = 0 \wmax$, there is no $k$ such that $w_{[k, \infty)} = \wmin$ either.  We conclude that $\wmin \less w_{[k, \infty)} \less \wmax$ for all $k \geq 1$.

Now consider the case $w_{[l, \infty)} \lesseq u$, which means that $w_{[k, \infty)} \lesseq u$ for all $k \geq 1$.   
Hence, $w_{[k, \infty)} \lesseq u \less \wmax$, where the second inequality follows from Equation \eqref{eq:u-ab}, and by the previous argument we also have $w_{[k, \infty)} \gess \wmin$ for all $k \geq 1$.  
\end{proof}

\begin{lem} \label{betab} Let  $v$ be a word for which there exists an index $l$ such that $v_{[l, \infty)} \gesseq v_{[k, \infty)}$ for all $k \geq 1$.  If $w$ is a word such that $w_{[k, \infty)} \lesseq v_{[l, \infty)}$ for all $k \geq 1$, then $\bar{\beta}(w) \leq \bar{\beta}(v)$.   On the other hand, if $w_{[i, \infty)} \gesseq v_{[l, \infty)}$ for some $i \geq 1$, then $\bar{\beta}(w) \ge \bar{\beta}(v)$. 
\end{lem} 

\begin{proof} To prove the first statement, suppose for contradiction that $\bar{\beta}(w) > \bar{\beta}(v)$ and take $\bar{\beta}(w) > \beta > \bar{\beta}(v)$.  Then $v \in \Wb$ and $v_{[l, \infty)} \lesseq \wmax$.   Since $w_{[k, \infty)} \lesseq v_{[l, \infty)}$ for all $k\ge 1$, we have $w_{[k, \infty)} \lesseq \wmax$ as well.  Hence, by Lemma \ref{lem:Wb}, $w \in \Wb$.  Therefore, $\bar{\beta}(w) \leq \beta$, a contradiction to our choice of $\beta$.

To prove the second statement, suppose now that $w_{[i, \infty)} \gesseq v_{[l, \infty)}$ for some $i$, and let $\beta > 1$ be such that $w \in \Wb$.  Then  $v_{[l, \infty)} \lesseq  w_{[i, \infty)} \lesseq \wmax$, using Lemma~\ref{lem:Wb}, and so $v_{[k, \infty)} \lesseq \wmax$ for all $k \geq 1$.  Therefore, $v \in \Wb$, and we conclude that $\bar{\beta}(w) \ge \bar{\beta}(v)$.
\end{proof}

By construction of $s^{(m)}$ and $t^{(m)}$, letting $w$ be one of these two words, we claim there is an index~$l$ such that $w_{[l, \infty)} \gesseq w_{[k, \infty)}$ for all $k \geq 1$.  To see this, recall that Equation~\eqref{eq:wmaxsubwords} states that $\wmax$ is greater than or equal to all of its subwords.  Since both $s^{(m)}$ and $t^{(m)}$ end in $\wmax$, they have only finitely many subwords that could potentially be greater than $\wmax$. Choosing the maximum among $\wmax$ and these subwords gives the index $l$.

Moreover, if $w$ is an eventually periodic word, then $w$ also has an index $l$ such that $w_{[l, \infty)} \gesseq w_{[k, \infty)}$ for all $k \geq 1$.  Indeed, $w$ has only finitely many distinct subwords in this case, so we may take $w_{[l, \infty)}$ to be the maximum of these with respect to $\gess$.  All of the words that we consider in this section will either be subwords of $s^{(m)}$ or $t^{(m)}$, or else eventually periodic. Thus, they will automatically satisfy this condition.  

For convenience, we define the following conditions, which will be used in the next few lemmas:
\begin{eqnarray}
\label{eq:smdefined} & n-\ell \text{ is odd and }  \beta > \b(\zeta p^\infty),\\
\label{eq:tmdefined} & \pi_n\neq 1 \text{ and } n-\ell \text{ is even and } \beta > \b(\zeta q^\infty).
\end{eqnarray}

Note that~\eqref{eq:smdefined} and~\eqref{eq:tmdefined} include the conditions for $s^{(m)}$ and $t^{(m)}$ to be defined respectively, and that condition~\eqref{eq:smdefined} implies that $\pi_n\neq n$.

\begin{lem} \label{lem:pwless} If \eqref{eq:smdefined} holds and $\wend$ is defined as in Equations~\eqref{eq:smdef} and~\eqref{eq:wend}, then $\wend \less p \wend$ and $\wend \less p^\infty$.  
Likewise, if \eqref{eq:tmdefined} holds and $\wend$ is defined as in Equations~\eqref{eq:tmdef} and~\eqref{eq:wendq}, then $q \wend \less \wend$ and $q^\infty \less \wend$.  \end{lem} 
  
\begin{proof} We will prove the first statement; the second statement can be proved similarly.
Since $\beta > \b(\zeta p^\infty)$, Lemma \ref{largeenough} applied to $w=\zeta p^\infty$ for $k = h$ gives $ \wmin \less (z_{[h, n-1]} z_{[x, h-1]})^\infty \less \wmax$. Then Lemma \ref{dinfinity} implies  
\begin{equation}\label{eq:pwless}   \wmin \less z_{[h, n-1]} z_{[x, h-1]} \wmin \quad \text{and}  \quad \wmax \gess z_{[h, n-1]} z_{[x, h-1]} \wmax.
\end{equation}

We will consider each case from Equations~\eqref{eq:smdef} and~\eqref{eq:wend} in turn.  In case (\ref{eq:smdef}.1), we verify $\wend = z_{[x, h-1]} \wmin \less p z_{[x, h-1]} \wmin = p \wend$ by canceling the even-length prefix $z_{[x, h-1]}$ obtaining the first part of Equation \eqref{eq:pwless}. 
In case (\ref{eq:smdef}.2), $\wend = p z_{[x, h-1]} \wmax \less p^2 z_{[x, h-1]} \wmax = p \wend$ holds by canceling the odd-length prefix $p z_{[x, h-1]}$ and using the second part of Equation \eqref{eq:pwless}. 
In case (\ref{eq:smdef}.3), we verify $\wend = z_{[x, h-1]} \wmax$ $\less z_{[x, h-1]} (z_{[h, n-1]} z_{[x, h-1]}) \wmax = p \wend$ by canceling the odd-length prefix $z_{[x, h-1]}$ and using the second part of Equation \eqref{eq:pwless}.
This proves that $\wend \less p\wend$ in the three cases in Equation~\eqref{eq:smdef}  and Lemma \ref{dinfinity} gives $\wend \less p^\infty$.   

Finally, in the special case described in Equation \eqref{eq:wend}, we have $\wend = z_{[x, h-1]} (z_{[j, h-1]})^\infty$.  To verify that $\wend \less p^\infty$, note that canceling even-length equal prefixes $z_{[x, h-1]}$,
the inequality is equivalent to \begin{equation}\label{eq:show} (z_{[j, h-1]})^\infty \less z_{[h, n-1]} p^\infty = (z_{[h, n-1]} z_{[x, h-1]})^\infty.\end{equation}  In Section \ref{sec:-N}, we showed that the word $s$ defined in Equation~\eqref{eq:defs} induces $\pi$, implying that $$s_{[j, \infty)} = z_{[j, n-1]} z_{[x, j-1]} s_{[n + j, \infty)} \less z_{[h, n-1]} z_{[x, h-1]} s_{[n + h, \infty)} = s_{[h, \infty)}.$$   Therefore, $z_{[j, n-1]} z_{[x, j - 1]} \lesseq z_{[h, n-1]} z_{[x, h-1]}$, and so $(z_{[j, n-1]}  z_{[x, j-1]})^\infty \lesseq  (z_{[h, n-1]}z_{[x, h-1]})^\infty$, \\ which is equivalent to $z_{[j, h-1]} (z_{[h, n-1]} z_{[x, h-1]})^\infty \lesseq (z_{[h, n-1]} z_{[x, h-1]})^\infty$. 
By Lemma \ref{dinfinity}, we conclude that $$(z_{[j, h-1]})^\infty \lesseq (z_{[h, n-1]} z_{[x, h-1]})^\infty.$$
We claim that this inequality is strict.  To see this, notice that the 
word $z_{[h, n-1]} z_{[x, h-1]}$ is either primitive or the square of a primitive word, say $d'$, of odd length because it is a cyclic shift of $p$, which has this property by Lemma \ref{d2}.  Since $h - j < n - x$, the equality $(z_{[j, h-1]})^\infty = (z_{[h, n-1]} z_{[x, h-1]})^\infty$ would imply $z_{[j, h-1]} = d'$ and $z_{[h, n-1]} z_{[x, h-1]} = d'^2$. This would only be possible if $j = x$ and $z_{[j, h-1]} = z_{[x, h-1]} = d'$, a case eliminated because $h-x$ was assumed to be even.  Equation~\eqref{eq:show} follows.  Lastly, by Lemma \ref{dinfinity}, $\wend \less p^\infty$ implies that $\wend \less p \wend$.

\end{proof}

\begin{lem} \label{betainduces}  If \eqref{eq:smdefined} holds and $m \geq \frac{n-1}{2}$, then  $\Pat(s^{(m)},\Sigma_{-},n)=\pi$.  Likewise, if \eqref{eq:tmdefined} holds and $m \geq \frac{n-1}{2}$, then  $\Pat(t^{(m)},\Sigma_{-} ,n)=\pi$.
 \end{lem}

\begin{proof} We will prove the statement for $s^{(m)}$; the proof for $t^{(m)}$ follows similarly.  Fix $m \geq \frac{n-1}{2}$  and let $s = s^{(m)}$. We will show that 
\begin{enumerate}[a)]
\item $s_{[x, \infty)} \gess s_{[n, \infty)}$,
\item there is no $1 \leq c \leq n$ such that $s_{[n, \infty)} \less s_{[c, \infty)} \less s_{[x, \infty)}$,
\item $\Pat(s, \Sigma_{-}, n)$ is defined.
\end{enumerate}    
Then Lemma \ref{pattern} can be applied to conclude that $s$ induces the pattern $\pi$.

To prove a), note that the first part of  Lemma \ref{lem:pwless} implies that $\wend \less p \wend$.  Therefore, $$s_{[n, \infty)} = p^{2m} \wend \less p^{2m+1} \wend = s_{[x, \infty)}.$$

To prove b), suppose for contradiction that $s_{[n, \infty)} \less s_{[c, \infty)} \less s_{[x, \infty)}$ for some $1 \leq c \leq n$, that is,
$$s_{[n, \infty)} = p^{2m} \wend \less s_{[c, \infty)} \less p^{2m+1} \wend = s_{[x, \infty)}.$$
This forces $s_{[c, \infty)} = p^{2m} v$ for some word $v$ satisfying 
\begin{equation}\label{eq:v} \wend \less v \less p \wend.\end{equation}  

Let us show that $c < x$.  If $p$ is primitive, this is because the first $p$ in $s_{[c, \infty)}$ cannot overlap with both the first and second occurrences of $p$ in $s_{[x, \infty)}$.  If $p$ is not primitive, then by Lemma \ref{d2}, $p = d^2$, where $d$ is primitive and $|d|$ is odd.  The only way to have $c > x$ would be if $v = d \wend$.  However, this is a contradiction with $\wend \less v \less d^2 \wend = p\wend$ because $d\wend \less d^2 \wend$ implies that $\wend \gess d \wend$.

To show that Equation~\eqref{eq:v} leads to a contradiction, consider two cases depending on whether $p$ is primitive or not:

\begin{itemize}
\item If $p$ is primitive, we claim that $v = p^i \wend$ for some $i \geq 2$. This is clear if $|p^{2m}| > n-1$ or $c>1$, because then one of the initial $2m$ occurrences of $p$ in $s_{[c, \infty)}$ must coincide with the first occurrence of $p$ in $s_{[x, \infty)}$. The other situation is when $|p| = 1$, $x = n-1$, $m = \frac{n-1}{2}$ and $c = 1$, but then $s_{[c, \infty)} = p^{n-1} s_{[x, \infty)}$ and $v = s_{[x, \infty)} = p^{2m+1} \wend$.

First suppose that $|p|$ is even.  Since $\wend \less p \wend$, we obtain $\wend \less p \wend \less p^2 \wend \less \cdots$.  Therefore, $p \wend \less p^i \wend = v$, contradicting Equation~\eqref{eq:v}.  

Now suppose that $|p|$ is odd.  The overlap between $s_{[c, \infty)}$ and $s_{[x, \infty)}$ causes $\zeta$ to be of the form $\zeta = app$.     By Lemma \ref{onecollapsed}, we have $q = p^2$, contradicting the fact that $\zeta$ was obtained from a valid $-N$-segmentation.    

\item If $p$ is not primitive, Lemma \ref{d2} implies that $p = d^2$ for some primitive word $d$ of odd length.   Since $|d^{2(2m)}| > n-1$, one of the initial $4m$ occurrences of $d$ in $s_{[c, \infty)}$ must coincide with the first occurrence of $d$ in $s_{[x, \infty)}$.  Therefore, $v = d^j \wend$ for some $j \geq 3$, using that $c < x$.  Moreover, by Lemma \ref{lem:pwless}, $\wend \less d^\infty$ and by Lemma \ref{dinfinity}, $\wend \less d^{j-2} \wend$.  It follows that $d^2 \wend \less d^{j} \wend = v$, contradicting Equation~\eqref{eq:v}.
\end{itemize}

Finally, to prove c), recall that $\wend < p^\infty$ by Lemma \ref{lem:pwless}.  Suppose for contradiction that
$s_{[i, \infty)} = s_{[j, \infty)}$ for some $i, j \leq n$ with $i \neq j$. Then $z_{[i, n-1]} p^{2m} \wend = z_{[j, n-1]} p^{2m} \wend$. If $p$ is primitive, we must have $\wend = p^k \wend$, and if not, by Lemma~\ref{d2}, we have $p = d^2$ and $\wend = d^k \wend$, for some $k \geq 1$.  In either case, this would imply $\wend = p^\infty$, contradicting Lemma \ref{lem:pwless}.
\end{proof}

\begin{lem} \label{lem:inWb} If \eqref{eq:smdefined} holds, there exists some $m_0 \geq \frac{n-1}{2}$ such that $s^{(m)} \in \Wb$ for all $m \ge m_0$.  Likewise, if \eqref{eq:tmdefined} holds, there exists some $m_0 \geq \frac{n-1}{2}$ such that $t^{(m)} \in \Wb$ for all $m \ge m_0$. \end{lem}

\begin{proof}  We will prove the statement about $s^{(m)}$; the one about $t^{(m)}$ follows by a similar argument.  

Assume that \eqref{eq:smdefined} holds. By Lemma \ref{largeenough} applied to $\zeta p^\infty$ for $k = \ell$ and $k = x$, we obtain \begin{equation} \label{eq:ellx} z_{[\ell, n-1]} p^\infty\less \wmax \text{ and } \wmin \less p^\infty \less \wmax. \end{equation}   Moreover, for any index $1 \leq j < n$, applying Lemma \ref{largeenough} to $\zeta p^\infty$ for $k = j$ gives
\begin{equation} \label{eq:largeenoughineq} \wmin \less z_{[j, n-1]} p^\infty \less \wmax. \end{equation}

 By Lemma \ref{lem:pwless},  $\wend \less p^\infty$, and so $\wend \less p^2 \wend$ by Lemma \ref{dinfinity}.  Hence, $p^{2(i-1)} \wend \less p^{2i} \wend \less p^\infty$ for all $i \geq 1$.  Since $n -\ell$ is odd,  $$z_{[\ell, n-1]} p^\infty \less z_{[\ell, n-1]} p^{2i} \wend \less z_{[\ell, n-1]}p^{2(i-1)}\wend$$ for all $i \geq 1$.  Choose $m_0 \geq \frac{n-1}{2}$ to be large enough such that $z_{[\ell, n-1]}p^{2m_0} \wend \less \wmax$, which is guaranteed to exist because of Equation~\eqref{eq:ellx}. 
 
Let $m \geq m_0$, and recall the definition $s^{(m)} = \zeta p^{2m}\wend$.  To show that $s^{(m)}\in \Wb$, we will show that
\begin{enumerate}[a)]
\item $z_{[i, n-1]}p^{2m} \wend \lesseq \wmax$ for all $1 \leq i < n$;
\item $z_{[i, n-1]} p^k \wend \lesseq \wmax$ for all $x \leq i < n$ and $0 \leq k < 2m$, in the cases described by Equation \eqref{eq:smdef};
\item  ${\wend}_{[i, \infty)} \lesseq \wmax$ for all $i \geq 1$, in the cases described by Equation \eqref{eq:smdef};
\item $z_{[i, n-1]}p^k \wend \lesseq \wmax$ for all $x \leq i < n$ and $0 \leq k < 2m$, and ${\wend}_{[i, \infty)} \lesseq \wmax$ for all $i \geq 1$, in the case described by Equation \eqref{eq:wend}. \end{enumerate} 
 
\medskip

First we prove a). Since $m \geq m_0$, we have $z_{[\ell, n-1]}p^{2m} \wend \lesseq z_{[\ell, n-1]} p^{2m_0} \wend \less \wmax$.  By Lemma~\ref{betainduces}, $s^{(m)}$ induces $\pi$, and we conclude that \begin{equation} \label{eq:induceineq} s^{(m)}_{[i, \infty)}=z_{[i, n-1]}p^{2m} \wend \lesseq z_{[\ell, n-1]} p^{2m} \wend \less \wmax \end{equation} for all $1 \leq i < n$.  

\medskip

Next we prove b). 
Since $s^{(m)}$ induces $\pi$, we have that $z_{[i, n-1]}p^{2m} \wend \less z_{[h, n-1]} p^{2m} \wend$ for all $x \leq i < n$, $i \neq h$.  Therefore, $z_{[i, n-1]} z_{[x, i -1]} \lesseq z_{[h, n-1]} z_{[x, h-1]}$ for each such $i$, and we consider two cases:
\begin{itemize} 
\item If $z_{[i, n-1]} z_{[x, i-1]} \less z_{[h, n-1]} z_{[x, h-1]}$, using that $\wend$ begins with $z_{[x, h-1]}$ in each of the cases in Equation \eqref{eq:smdef}, we conclude that, for all $0 \leq k < 2m$,
$$z_{[i, n-1]} p^k \wend \less z_{[h, n-1]}p^{2m} \wend \less \wmax,$$ where the second inequality follows from Equation \eqref{eq:induceineq}. 

\item If $z_{[i, n-1]} z_{[x, i-1]} = z_{[h, n-1]} z_{[x, h-1]}$, we also have $z_{[i, n-1]} p^\infty = z_{[h, n-1]} p^\infty$.  Since  $i\neq h$, we find that $p$ cannot be primitive. By Lemma \ref{d2}, $p = d^2$, where $d$ is primitive and $|d|$ is odd. It follows that $\{i,h\}=\{x,x+|d|\}$, leaving two possibilities.

If $x = h$, we are in case (\ref{eq:smdef}.1), and we have $\wend =  \wmin$ and $z_{[h, i-1]} = z_{[i, n-1]} = d$. By Equation \eqref{eq:ellx}, $d^\infty = p^\infty \gess \wmin$, and Lemma \ref{dinfinity} implies that $d \wmin \gess \wmin$, that is, $d \wend \gess \wend$.  
Adding the odd-length prefix $d^{2m+1}$ to both sides, we get $s_{[h, \infty)} = d^{2m + 2} \wend \less d^{2m+1} \wend = s_{[i, \infty)}$.  This is a contradiction since $\pi_h > \pi_i$ by definition of $h$, and $s^{(m)}$ induces $\pi$ by Lemma~\ref{betainduces}. 

If $x = i$, we are in case (\ref{eq:smdef}.3) since $h-x=h-i$ is odd. We have $\wend = d \wmax$ and $z_{[i, h-1]} = z_{[h, n-1]} = d$, and so $z_{[i, n-1]}p^k \wend = d^{2k + 2} \wmax$.  Since $p^\infty = d^\infty \less \wmax$ by Equation \eqref{eq:ellx}, we obtain by Lemma \ref{dinfinity} that $d^j \wmax \less \wmax$ for all $j \geq 1$, proving b).
\end{itemize}

\medskip

To show c),  we must verify that $\wend \in \Wb$ in each of the cases defined by Equation \eqref{eq:smdef}.  Recall that  ${\wmax}_{[j, \infty)} \lesseq \wmax$ for all $j \geq 1$ by Equation~\eqref{eq:wmaxsubwords} . 

In case (\ref{eq:smdef}.1), we have $\wend = z_{[x, h-1]} \wmin$.   Let $i\ge1$, and suppose first that $h-i$ is even.  We verify that $z_{[i, h-1]} \wmin \less z_{[i, n-1]} p^\infty$ by canceling $z_{[i, h-1]}$ to obtain Equation \eqref{eq:largeenoughineq} for $j = h$.  Moreover, we have $z_{[i, n-1]} p^\infty \less \wmax$ by Equation \eqref{eq:largeenoughineq}, allowing us to conclude that $z_{[i, h-1]} \wmin \less z_{[i, n-1]} p^\infty \less \wmax$.   Now suppose that $h- i$ is odd.  As in case b), since $\pi_i < \pi_{\ell}$, we have $z_{[i, n-1]}p^{2m} \wend \less z_{[\ell, n-1]}p^{2m} \wend$, which implies that $z_{[i, h-1]} \lesseq z_{[\ell, \ell + h - i - 1]}$.  If this inequality is strict, then $z_{[i, h-1]} \wmin \less z_{[\ell, n-1]} p^\infty \less \wmax$ by Equation~\eqref{eq:ellx}.  Otherwise, this gives $z_{[i , h-1]} = z_{[\ell, \ell + h - i - 1]}$ and $i$ would be an index satisfying the conditions of the exceptional case in Equation \eqref{eq:wend}.

In case (\ref{eq:smdef}.2), we defined $\wend = p z_{[x, h-1]}\wmax$.    We will begin by verifying that $z_{[i, h-1]} \wmax \less \wmax$ for each $x \leq i < h$. Since $\pi_i < \pi_h$, Lemma \ref{flipping} implies that $z_{[i, n-1]} z_{[x, i-1]} \lesseq z_{[h, n-1]} z_{[x, h-1]}$.  Therefore, $(z_{[i, n-1]} z_{[x, i-1]})^\infty \lesseq (z_{[h, n-1]} z_{[x, h-1]})^\infty$, which can be rewritten as
\begin{equation} \label{eq:wendless}  z_{[i, h-1]} (z_{[h, n-1]} z_{[x, h-1]})^\infty \lesseq  (z_{[h, n-1]} z_{[x, h-1]})^\infty. \end{equation} By Lemma \ref{dinfinity}, we obtain $(z_{[i, h-1]})^\infty \lesseq (z_{[h, n-1]} z_{[x, h-1]})^\infty \less \wmax$, where the second inequality follows from Equation \eqref{eq:largeenoughineq} for $j = h$.  Applying Lemma \ref{dinfinity} to $(z_{[i, h-1]})^\infty \less \wmax$ gives
\begin{equation} \label{eq:zihwmax} z_{[i, h-1]}\wmax \less \wmax. \end{equation} 
The remaining subwords of $\wend$ are of the form $z_{[i, n-1]} z_{[x, h-1]} \wmax$, where $x \leq i < h$, and we must now verify that $z_{[i, n-1]} z_{[x, h-1]} \wmax \less \wmax$ holds for each index $i$. 
As in the proof of case b), the equality $z_{[i, n-1]} z_{[x, i-1]} = z_{[h, n-1]} z_{[x, h-1]}$ would imply that 
$p = d^2$, which is impossible because $|p|$ is odd. Thus, $z_{[i, n-1]} z_{[x, i-1]} \less z_{[h, n-1]} z_{[x, h-1]}$.  
It follows that
$$z_{[i, n-1]} z_{[x, h-1]} \wmax \less z_{[h, n-1]} p^{2m} \wend \less \wmax,$$
where the second inequality comes from Equation \eqref{eq:induceineq}.  

In case (\ref{eq:smdef}.3), we defined $\wend = z_{[x, h-1]} \wmax$.  Since the proof of Equation \eqref{eq:zihwmax} does not depend on the parity of $h-x$, the same argument shows that $z_{[i, h-1]} \wmax \less \wmax$ for each $x \leq i < h$ in this case.

\medskip

To show d), let us consider the exceptional case defined in Equation \eqref{eq:wend}, in which $h-x$ and $|p|$ are both even and there is an index $x \le j < h$ such that $h-j$ is odd and $z_{[j, h-1]} = z_{[\ell, \ell + h - j - 1]}$.  We defined $\wend =  z_{[x, h -1]} (z_{[j, h-1]})^\infty$, where $\pi_j > \pi_{j'}$ for any other indices $j'$ satisfying this property. We will first verify that $z_{[i, h-1]}(z_{[j, h-1]})^\infty \less \wmax$ for indices $x \le i < h$, thus showing $\wend \in \Wb$. 

As in case b), $\pi_j < \pi_h$ implies that $z_{[j, n-1]} p^{2m} \wend \less z_{[h, n-1]} p^{2m} \wend$, and so $z_{[j, n-1]} z_{[x, j-1]} \lesseq z_{[h, n-1]} z_{[x, h-1]}$. Therefore, $(z_{[j, n-1]} z_{[x, j-1]})^\infty \lesseq (z_{[h, n-1]} z_{[x, h-1]})^\infty$, and by Lemma \ref{dinfinity} we have
\begin{equation} \label{eq:exception} (z_{[j, h-1]})^\infty \lesseq (z_{[h, n-1]} z_{[x, h-1]})^\infty. \end{equation} 
  Furthermore,  this implies that $z_{[j, n-1]} p^\infty \lesseq (z_{[j, h-1]})^\infty$, which is verified by canceling the odd-length prefix $z_{[j, h-1]}$ to obtain Equation \eqref{eq:exception}.  
  
Since we chose $\pi_j > \pi_{j'}$ for any other indices $j'$ satisfying $z_{[j', h-1]} = z_{[\ell, \ell + h - j -1]}$, either \\ $z_{[i, i + h - j - 1]} \less z_{[j, h-1]}$ or $\pi_i < \pi_j$.  In the first case,  the inequality $z_{[i, n-1]}p^{2m} \wend \less z_{[j, n-1]} p^{2m} \wend$ follows immediately.  Otherwise, as in case b), the fact that $s^{(m)}$ induces $\pi$ and $\pi_i < \pi_j$ implies that $z_{[i, n-1]}p^{2m} \wend \less z_{[j, n-1]} p^{2m} \wend$.  We now have $z_{[i, n-1]} z_{[x, i-1]} \lesseq z_{[j, n-1]} z_{[x, j-1]}$, hence $(z_{[i, n-1]}z_{[x, i-1]})^\infty \lesseq (z_{[j, n-1]}z_{[x, j-1]})^\infty = z_{[j, n-1]}p^\infty$ and Lemma \ref{dinfinity} implies that
$$(z_{[i, j-1]})^\infty \lesseq z_{[j, n-1]} p^\infty \lesseq (z_{[j, h-1]})^\infty.$$
Lastly, by Lemma \ref{dinfinity}, we obtain $z_{[i, j-1]} (z_{[j, h-1]})^\infty \lesseq (z_{[j, h-1]})^\infty \lesseq (z_{[h, n-1]} z_{[x, h-1]})^\infty \less \wmax$, where the second inequality holds by Equation \eqref{eq:exception} and the third inequality comes from Equation \eqref{eq:largeenoughineq} taken at the index $h$. Hence, $z_{[i, h-1]} (z_{[j, h-1]})^\infty \less \wmax$ for each $x \le i < j$, thus $\wend \in \Wb$ in all cases.   
Now we must verify that $z_{[i, n-1]} p^k \wend \less \wmax$ for all $x \le i < n$ and $0 \leq k < 2m$.  Since $\pi_i < \pi_h$, we have $z_{[i, n-1]} z_{[x, i-1]} \lesseq z_{[h, n-1]}z_{[x, h-1]}$.  If this inequality is strict, it is immediate that $z_{[i, n-1]} p^k \wend \less (z_{[h, n-1]} z_{[x, h-1]})^\infty \less \wmax$, because $z_{[i, n-1]}p^k \wend$ begins with $z_{[i, n-1]} z_{[x, i-1]}$ for each $k$.  Otherwise, as in the proof of case b), the equality $z_{[i, n-1]}z_{[x, i-1]} = z_{[h, n-1]} z_{[x, h-1]}$ implies that $p = d^2$, where $|d|$ is odd, and $\{i, h\} = \{x, x + |d|\}$. The fact that Equation \eqref{eq:exception} implies that $x \neq h$ in this case tells us that $i = j = x$, $h = x + |d|$, and so $z_{[j, h-1]} = d$ and $z_{[i, n-1]} = p$.  Therefore, $\wend = d^\infty$ and $z_{[i, n-1]} p^k \wend = d^\infty \less \wmax$ follows by Equation \eqref{eq:ellx}.

We have now verified that $s^{(m)}_{[i, \infty)} \lesseq \wmax$ for all $i \geq 1$ and we conclude that $s^{(m)} \in \Wb$.
\end{proof}

Putting Lemmas \ref{betainduces} and \ref{lem:inWb} together we obtain the following.

\begin{prop} \label{prop:pat} If \eqref{eq:smdefined} holds, then there exists an $m_0 \geq \frac{n-1}{2}$ such that $s^{(m)}\in\Wb$ and $\Pat(s^{(m)}, \Sigma_{-\beta}, n) = \pi$ for all $m \geq m_0$.   If \eqref{eq:tmdefined} holds, then there exists an $m_0 \geq \frac{n-1}{2}$ such that $t^{(m)}\in\Wb$ and $\Pat(t^{(m)}, \Sigma_{-\beta}, n) = \pi$ for all $m \geq m_0$.  \end{prop}

The following proposition will allow us to express the values of $\b(\zeta p^\infty)$ and $\b(\zeta q^\infty)$, when they appear, as the largest real solution to a certain equation.  In Section \ref{sec:B}, we will rearrange the equation and see that the values of $\b(\zeta p^\infty)$ and $\b(\zeta q^\infty)$ may be interpreted as the largest real solution to a polynomial with integer coefficients.  

\begin{prop} \label{prop:bb} If $n - \ell$ is odd, then $z_{[\ell, n-1]} p^\infty$ is the largest subword of $\zeta p^\infty$ and 
 $\b(\zeta p^{\infty})$ is the largest real solution to $f_{z_{[\ell, n-1]}p^\infty}(\beta) = 1$ when $z_{[\ell, n-1]}p^\infty \gess u$, otherwise $\b(\zeta p^\infty) = 1$.   Likewise, if  $n - \ell$ is even and $\pi_n \neq 1$, then $z_{[\ell, n-1]} q^\infty$ is the largest subword of $\zeta q^\infty$ and 
 $\b(\zeta q^\infty)$ is the largest real solution to $f_{z_{[\ell, n-1]}q^\infty}(\beta) = 1$ when $z_{[\ell, n-1]}q^\infty \gess u$, otherwise $\b(\zeta q^\infty) = 1$. 
\end{prop} 

\begin{proof}  We will prove the statement for $\b(\zeta p^\infty)$; the one for $\b(\zeta q^\infty)$ follows similarly.  Since Lemma \ref{betainduces}  implies that $s^{(m)}$ induces $\pi$ for any $m \geq \frac{n-1}{2}$ and $\beta > \b(\zeta p^\infty)$, we have $z_{[\ell, n-1]}p^{2m} \wend \gess z_{[i, n-1]}p^{2m} \wend$ for all $1 \leq i < n$, $i \neq \ell$.   Therefore, $z_{[\ell, n-1]}p^{\infty} \gesseq z_{[i, n-1]}p^\infty$ for all $1 \leq i \leq n$.  
Since all subwords of $\zeta p^\infty$ are of the form $z_{[i, n-1]}p^\infty$ for some $1 \leq i \leq n-1$, we conclude that $z_{[\ell, n-1]}p^\infty$ is the maximal  subword of $\zeta p^\infty$.  By Definition \ref{barb}, $\b(\zeta p^\infty)$ is equal to the largest real solution to $f_{z_{[\ell, n-1]}p^\infty}(\beta) = 1$ in the case that $z_{[\ell, n-1]}p^\infty \gess u$ and is equal to $1$ otherwise.   \end{proof}

In the following propositions, we construct a word $w$ such that $\B(\pi) = \bar{\beta}(w) = \b(w)$.  The constructions depend on features of $\pi$ such as the parity of $n - \ell$ and whether $\pi$ is regular, cornered or collapsed.  The proposition associated to each type of permutation is listed in Table~\ref{tab:cases}.  The last column contains the word $w$ such that $\B(\pi) = \b(w)$ defined by these propositions.

\begin{table}[h]
\begin{tabular}{|c|c|l|c|}
\hline
\multirow{3}{*}{$\pi$ regular} &  $n-\ell$  is odd & Proposition~\ref{prop:regularnlodd} & $w  = \zeta p^\infty$ \\
\cline{2-4}
 &  $n-\ell$ is even and $\pi_n=1$ & Proposition~\ref{prop:regularnone} & $w = \zeta (0 z_{[\ell, n-1]} )^\infty$ \\
\cline{2-4}
 &  $n-\ell$ is even and $\pi_n \neq 1$ & Proposition~\ref{prop:regularnleven} & $w = \zeta q^\infty$ \\
\hline
\multirow{2}{*}{$\pi$ cornered} & $\pi_{n-2}\pi_{n-1}\pi_n = (n{-}1)1n$ & Proposition~\ref{prop:corneredn} & $w = \zeta q^\infty = \zeta ((N{-}2)0)^\infty$ \\
\cline{2-4} & $\pi_{n-2}\pi_{n-1}\pi_{n} = 2n1$ & Proposition~\ref{prop:cornered2} & $w = \zeta p^\infty = \zeta (0(N{-}2))^\infty$ \\
\hline

\multirow{4}{*}{$\pi$ collapsed} & \multirow{2}{*}{$n-\ell$ is odd} & \multirow{2}{*}{Proposition~\ref{prop:collapsednlodd}} &  $w = \zeta^{(k)} (p^{(k)})^\infty$, \\
& & & $\zeta^{(k)}$ minimizing $z_{[\ell, n-1]}^{(k)} p^{(k)}$ \\
\cline{2-4} & \multirow{2}{*}{$n-\ell$ is even} & \multirow{2}{*}{Proposition~\ref{prop:collapsednleven}} & $w = \zeta^{(k)} (q^{(k)})^\infty$, \\
& & & $\zeta^{(k)}$ minimizing $z_{[\ell, n-1]}^{(k)} q^{(k)}$ \\
\hline
\end{tabular}
\caption{The different cases for $\pi$ and the propositions that construct $w$ with $\B(\pi)=\b(w)$ in each case.}
\label{tab:cases}
\end{table}

\begin{prop} \label{prop:regularnlodd}  Let $\pi$ be a regular permutation  such that $n-\ell$ is odd. 
Then
$$\B(\pi) = \b(\zeta p^\infty).$$  \end{prop}

\begin{proof} To show $\B(\pi) \leq \b(\zeta p^\infty)$, let $\beta > \b(\zeta p^\infty)$. By Proposition \ref{prop:pat}, there exists an $m_0 \geq \frac{n-1}{2}$ such that for all $m \geq m_0$, we have $s^{(m)}\in\Wb$ and $\Pat(s^{(m)}, \Sigma_{-\beta}, n ) = \pi$.  Thus, $\pi \in \Al(\Sigma_{-\beta})$ and $\B(\pi) \leq\beta$. 

To show $\B(\pi) \geq \b(\zeta p^\infty)$, let $w = \zeta w_{[n, \infty)} \in \mathcal{W}_N$ be a word inducing $\pi$.  Then we have $w_{[n, \infty)} \less w_{[x, \infty)}=pw_{[n, \infty)}$, and so, by Lemma \ref{dinfinity}, $w_{[n, \infty)} \less p^\infty$.  Since $n-\ell$ is odd, we obtain $w_{[\ell, \infty)} = z_{[\ell, n-1]} w_{[n, \infty)} \gess z_{[\ell, n-1]} p^\infty$. Moreover, Lemma \ref{betab} implies that $\bar{\beta}(w) \ge \bar{\beta}(\zeta p^\infty)$ because Proposition \ref{prop:bb} states that $z_{[\ell, n-1]} p^\infty$ is the largest subword of $\zeta p^\infty$.  Since Proposition \ref{betcor} tells us that $\bar{\beta}(\zeta p^\infty) = \b(\zeta p^\infty)$, we may now conclude that $\B(\pi) \geq \b(\zeta p^\infty)$. 
\end{proof}

\begin{example}  Let $\pi = 516324$, a regular permutation such that $n- \ell=6-3$ is odd.  Then $\hat{\pi} = 642\star13$, $\asc(\hat{\pi}) = 1$ and $\N(\pi) =2$.  A $-2$-segmentation of $\hat{\pi}$ is given by $(e_0, e_1, e_2) = (0, 5, 6)$, defining the prefix $\zeta = 00100$.  Since $h - x = \ell - x=3-1$ is even and $p = 00100$ has odd length, we use (\ref{eq:smdef}.2) to obtain $\wend = p z_{[x, h-1]} \wmax = p 00  \wmax$.  By Proposition \ref{prop:regularnlodd}, $\B(\pi) = \b(\zeta p^\infty) \approx 1.466$, where the value of $\b(\zeta p^\infty)$ is determined by Theorem \ref{thm:P}. Hence,  by Lemma \ref{betainduces}, $s^{(m)} = 00100(00100)^{2m+1}00 \wmax$ induces $\pi$ for all $m \geq 3$ and $\beta > \b(\zeta p^\infty)$.  Moreover, Lemma \ref{lem:inWb} guarantees that given such $\beta$ there exists an $m_0 \geq 3$ such that, for any $m \geq m_0$, we have $s^{(m)} \in \Wb$ and $\Pat(s^{(m)}, \Sigma_{-\beta}, n) = \pi$.  \end{example}

\begin{prop} \label{prop:regularnone} Let $\pi$ be a regular permutation such that $n-\ell$ is even and $\pi_n = 1$. 
Suppose that $\beta > \b(\zeta 0(z_{[\ell, n-1]})^\infty)$, and let
$$w = \zeta \wmin.$$
Then $w \in \Wb$, $w$ induces $\pi$, and 
$$\B(\pi) = \b(\zeta (0 z_{[\ell, n-1]})^\infty).$$
\end{prop}

\begin{proof}  We will show that $w \in \Wb$, $\Pat(w, \Sigma_{-\beta}, n) = \pi$ and $\B(\pi) = \b(\zeta(0z_{[\ell, n-1]})^\infty)$.  We begin by noting that Lemma \ref{largeenough} applied to $\zeta (0 z_{[\ell, n-1]})^\infty$ for $k = \ell$ implies that $\wmin \less (z_{[\ell, n-1]}0)^\infty \less \wmax$, and Lemma \ref{dinfinity} gives 
\begin{equation} \label{ellnone} z_{[\ell, n-1]} 0 \wmax = z_{[\ell, n-1]} \wmin \less \wmax.\end{equation}  

We will first show that $\Pat(w, \Sigma_{-}, n) = \pi$.  As in Section \ref{sec:-N}, this will follow by showing that 
\begin{enumerate}[a)]
\item there is no $1 \leq c < n$ such that $w_{[c, \infty)} \less w_{[n, \infty)} = \wmin$,
\item $\Pat(w, \Sigma_{-}, n)$ is defined, and
\item $w$ induces $\pi$, which we do by following the proof of Lemma \ref{pattern}.  
\end{enumerate}

Notice that if there were an index $1 \leq c < n$ such that $w_{[c, \infty)} \lesseq \wmin$, then $w_c = 0$ and $w_{[c+ 1, \infty)} \gesseq \wmax$.  To prove a), suppose for a contradiction that $w_{[c, \infty)} \gesseq \wmax$ for some $2 \leq c \leq n$.  By Equation~\eqref{ellnone}, it follows that $c \neq \ell$ and also that $z_{[c, n-1]} \wmin \gesseq z_{[\ell, n-1]} \wmin$.   If $c < \ell$, then Lemma \ref{flipping} implies that $z_{[c, c + n - \ell - 1]} \lesseq z_{[\ell, n-1]}$, hence $z_{[c, c + n - \ell -1]} = z_{[\ell, n-1]}$.  Applying Lemma \ref{flipping} $n-\ell$ times to $\pi_{\ell} > \pi_{c}$ gives $1 = \pi_n > \pi_{c + n - \ell}$, a contradiction.  If $c > \ell$, then Lemma \ref{flipping} implies that $z_{[c, n-1]} \lesseq z_{[\ell, \ell + n - c -1]}$, thus $z_{[c, n-1]} = z_{[\ell, \ell + n - c -1]}$.   Since $\pi_{\ell} > \pi_{\ell + n -c}$, applying Lemma \ref{flipping} $n - c$ times implies that $n - c$ is even, since otherwise we would obtain $\pi_{\ell + n - c} < \pi_{n} =1$, a contradiction.  Hence, if such an index $c$ exists, we must have $c > \ell$, $n - c$ is even and $z_{[c, n-1]} = z_{[\ell, \ell + n - c - 1]}$.  But then it would follow that $w_{[c, \infty)} = z_{[c, n-1]} \wmin \lesseq z_{[\ell, n-1]} \wmin = w_{[\ell, \infty)}$, a contradiction to the fact that $w_{[c, \infty)} \gesseq \wmax \gess w_{[\ell, \infty)}$.

To prove b), note that for $i, j \leq n$, the equality $w_{[i, \infty)} = w_{[j, \infty)}$ together with a) implies that these two words have the first instance of $\wmin$ appearing at the same position.  By a), there is no $1 \leq c < n$ such that $w_{[c, \infty)} = \wmin$, which forces $i = j$ in the previous equality.   Therefore, $\Pat(w, \Sigma_{-}, n)$ is defined.

Finally, to prove c), for $1 \leq i, j \leq n$, let $S(i, j)$ be the statement 
$$\pi_i < \pi_j \text{ implies }  w_{[i, \infty)} \less w_{[j, \infty)}.$$
To show $\Pat(w, \Sigma_{-}, n) = \pi$, we will prove $S(i, j)$ for all $1 \leq i, j \leq n$, with $i \neq j$.   We begin by verifying the endpoint cases, which we will often reduce to in the remainder of the proof.  If $j = n$, the statement $S(i, n)$ holds because we never have $\pi_i < \pi_n = 1$.  If $i = n$, the statement $S(n, j)$ holds because part a) implies there is no $1 \leq c \leq n$ such that $w_{[c, \infty)} \less w_{[n, \infty)}$.

We are left with the case when $1 \leq i, j < n$.  Suppose that $\pi_i < \pi_j$.  Let $m$ be such that $w_{[i, i + m - 1]} = w_{[j, j + m-1]}$ and $w_{i + m} \neq w_{j + m}$.  First assume that $i + m, j + m \leq n-1$.  If $m$ is even, then Lemma \ref{flipping} applied $m$ times to $\pi_i < \pi_j$ implies that $\pi_{i + m} < \pi_{j + m}$.  By Lemma \ref{flipping}, we must have $w_{i + m} \leq w_{j + m}$, thus $w_{i + m} < w_{j + m}$.  Therefore, $w_{[i + m, \infty)} \less w_{[j + m, \infty)}$ and  we obtain $w_{[i, \infty)} \less w_{[j, \infty)}$.  Similarly, if $m$ is odd, then Lemma \ref{flipping} applied $m$ times to $\pi_i < \pi_j$ implies that $\pi_{i + m} > \pi_{j + m}$.  Hence, by Lemma \ref{flipping}, we must have $w_{i + m} > w_{j + m}$ because $w_{i + m} \neq w_{j + m}$.  Therefore, $w_{[i + m, \infty)} \gess w_{[j + m, \infty)}$, and thus $w_{[i, \infty)} \less w_{[j, \infty)}$.  This shows that whenever $i + m, j + m \leq n -1$, the statement $S(i, j)$ holds.

Suppose now that $i + m \geq n$ or $j + m \geq n$ and let $m'$ be the minimal index such that either $i + m' = n$ or $j + m' = n$. 
\begin{itemize}
\item  If $i + m' = n$ and $m'$ is even, then $w_{[i, i + m' - 1]} = w_{[j, j + m' -1]}$ and Lemma \ref{flipping} implies that $\pi_n = \pi_{i + m'} < \pi_{j + m'}$.  In the first paragraph of the proof of part c), we have shown that $S(n, j + m')$ holds.  Thus, $w_{[n, \infty)} \less w_{[j + m', \infty)}$ and we obtain $w_{[i, \infty)} \less w_{[j, \infty)}$.  We conclude that $S(i, j)$ holds.  

\item If $i + m' = n$ and $m'$ is odd, then $w_{[i, i + m' - 1]} = w_{[j, j + m' - 1]}$ and Lemma \ref{flipping}  implies that $1 =\pi_{n} = \pi_{i + m'} > \pi_{j + m'}$, a contradiction. 

\item  If $j + m' = n$ and $m'$ is even, and $w_{[i, i + m' - 1]} = w_{[j, j + m' -1]}$and Lemma \ref{flipping} implies that $\pi_{i + m'} < \pi_{j + m'} = \pi_n = 1$, a contradiction.  

\item If $j + m' = n$ and $m'$ is odd, then  $w_{[i, i + m'-1]} = w_{[j, j + m' -1]}$ and Lemma \ref{flipping} implies that $\pi_{i + m'} > \pi_{j + m'} = \pi_{n} = 1$.  In the first paragraph of the proof of part c), we verified that $S(i + m', n)$ holds.  Thus, $w_{[i + m', \infty)} \gess w_{[n, \infty)}$ and we obtain $w_{[i, \infty)} \less w_{[j, \infty)}$.   We conclude that $S(i, j)$ holds.  
\end{itemize}

We have shown that $S(i,j)$ holds in all cases, hence $\Pat(w, \Sigma_{-}, n) = \pi$.  Moreover, $w \in \Wb$ because $w_{[i, \infty)} \lesseq w_{[\ell, \infty)}$ for all $1 \leq i < n$ and $w_{[\ell, \infty)} = z_{[\ell, n-1]} \wmin \less \wmax$ by Equation~\eqref{ellnone}. 

The above argument shows that for any $\beta > \b(\zeta(0z_{[\ell, n-1]})^\infty)$, we have $w = \zeta \wmin \in \Wb$ and $\Pat(w, \Sigma_{-\beta}, n) = \pi$. It follows that $\B(\pi) \leq \b(\zeta(0z_{[\ell, n-1]})^\infty)$

Finally, to show that $\B(\pi) \geq \b(\zeta(0z_{[\ell, n-1]})^\infty)$, let $ \b(\zeta(0z_{[\ell, n-1]})^\infty) > \beta > 1$.  We will show that no word in $\Wb$ induces $\pi$.  Suppose for contradiction that $v \in \Wb$ induces $\pi$. Lemma \ref{beginszeta}  implies that $v$ begins with $\zeta$, that is, $v= \zeta v_{[n, \infty)}$.  By Lemma \ref{lem:Wb}, we have $\wmax \gesseq v_{[\ell, \infty)}$ and $ v_{[n, \infty)} \gesseq \wmin$.  The fact that  $n - \ell$ is even implies that $v_{[\ell, \infty)} = z_{[\ell, n-1]} v_{[n, \infty)} \gesseq z_{[\ell, n-1]} \wmin$.  Therefore, $\wmax \gesseq z_{[\ell, n-1]} 0 \wmax$, and by Lemma~\ref{dinfinity}, $\wmax \gesseq (z_{[\ell, n-1]} 0)^\infty$.  It follows that $\beta \geq \bar{\beta}(\zeta(0 z_{[\ell, n-1]})^\infty) = \b(\zeta (0 z_{[\ell, n-1]})^\infty)$, where the equality comes from Lemma~\ref{betcor},  a contradiction to our choice of $\beta$.  It follows that  $\B(\pi) \ge \b(\zeta(0 z_{[\ell, n-1]})^\infty)$.  
\end{proof}

\begin{example} \label{nlevenone} Let $\pi = 6435721$, a regular permutation such that $n-\ell$ is even and $\pi_n = 1$. Then $\hat{\pi} = {\star} 153742$,  $\asc(\hat{\pi}) = 2$ and $\N(\pi) = 3$.  A $-3$-segmentation of $\hat{\pi}$ is given by $(e_0, e_1, e_2, e_3) = (0, 2, 4, 7)$, defining the prefix $\zeta = 211220$.  Therefore, by Proposition \ref{prop:regularnleven}, $w = 211220 \wmin$ and $\Pat(w, \Sigma_{-\beta}, n) = \pi$ for all $\beta > \b(\zeta (0z_{[\ell, n-1]})^\infty) = \frac{3 + \sqrt{5}}{2} \approx 2.618$. \end{example}

\begin{prop} \label{prop:regularnleven}
Let $\pi$ be a regular permutation such that $n-\ell$ is even and $\pi_n \neq 1$. 
Then
$$\B(\pi) = \b(\zeta q^\infty).$$ 
   \end{prop}

\begin{proof} The proof follows the same argument as Proposition \ref{prop:regularnlodd}. \end{proof}

\begin{example} \label{nleven} 
Let $\pi = 15237864$, a regular permutation with $n - \ell$ even and $\pi_n \neq 1$.  Then $\hat{\pi} = 537\star2486$, $\asc(\hat{\pi}) = 3$ and $\N(\pi) = 4$.  A $-4$-segmentation of $\hat{\pi}$ is given by $(e_0, e_1, e_2, e_3, e_4) = (0, 2, 5, 6, 8)$, defining the prefix $\zeta = 0101332$.  Since $|q|$ is even, we use (\ref{eq:tmdef}.3) with $q =  1332$ and $\wend = z_{[y, h-1]}\wmax = 1 \wmax$.  From this, by Lemma \ref{betainduces} we obtain $t^{(m)} = 0101332(1332)^{2m} 1 \wmax$, and $t^{(m)}$ induces $\pi$ for all $\beta > \b(\zeta q^\infty)$ and $m \geq 4$. Moreover, by Lemma \ref{lem:inWb}, given a $\beta > \b(\zeta q^\infty) \approx 3.1544$, there exists an $m_0 \geq 4$ such that for any $m \geq m_0$, we have $t^{(m)} \in \Wb$ and $\Pat(t^{(m)}, \Sigma_{-\beta}, n) = \pi$. 
\end{example}
 
For the next to propositions, recall that we defined $N=\N(\pi)$. 

\begin{prop} \label{prop:corneredn} Let $\pi$ be a cornered permutation such that  $\pi_{n-2} \pi_{n-1} \pi_{n} = (n{-}1) 1 n$.  
Then
$$\B(\pi) = \b(\zeta q^\infty) = N-1,$$
where $\zeta$ is the unique prefix defined by a $-N$-segmentation with $e_{N-1} \geq n-1$.  
\end{prop}

\begin{proof}  
First let us calculate $\b(\zeta q^\infty)$.  Since $\pi$ is a cornered permutation, by the observation directly following Lemma \ref{numberseg}, the unique prefix $\zeta$ associated to a $-N$-segmentation of $\hat{\pi}$ with $e_{N-1} \geq n-1$ satisfies $q = (N{-}2)0$, and $\zeta \in \{0, 1, \dots, N{-} 2 \}^{n-1}$.  By Lemma \ref{lem:unique_largest}, we may calculate $\b(\zeta q^\infty)$ directly by finding largest real solution to $f_{((N{-}2)0)^\infty)}(x) = 1$ and obtain $\b(\zeta q^\infty) = N-1$.  

By Theorem \ref{negativeshift}, there is no word $w \in \mathcal{W}_{N{-}1}$ inducing $\pi$.  Hence $\B(\pi) \geq {N{-}1} = \b(\zeta q^
\infty)$.  

To show $\B(\pi) \leq \b(\zeta q^\infty)$, let $\beta > \b(\zeta q^\infty)$.   By Proposition \ref{prop:pat}, there exists an $m_0 \geq \frac{n-1}{2}$ such that for all $m \geq m_0$,  we have $t^{(m)} \in \Wb$ and $\Pat(t^{(m)}, \Sigma_{-}, n) = \pi$.
 \end{proof}

\begin{example} \label{infismin} Let $\pi = 23654718$, a cornered permutation satisfying $\pi_{n-2} \pi_{n-1} \pi_{n} = (n-1)1 n$.  Then $\hat{\pi} = 8367451{\star}$, $\asc(\hat{\pi}) = 3$ and $\N(\pi) = 5$.  A valid $-5$ segmentation of $\hat{\pi}$ such that $e_{4} \geq 7$ is given by $(e_0, e_1, e_2, e_3, e_4, e_5) = (0, 2, 3, 5, 8, 8)$, defining the prefix $\zeta = 0132230$ with $q = 30$.  By Proposition \ref{prop:corneredn}, $\B(\pi) = 4$.  Moreover, by Lemmas~\ref{betainduces} and~\ref{lem:inWb}, $t^{(m)} = 0132230(30)^{2m} \wmax\in \Wb$ and $\Pat(t^{(m)}, \Sigma_{-\beta}, n) = \pi$ for all $\beta > 4$ and $m \geq m_0 = 4$.  In this example one can take $m_0 = 4$ because, by inspection, $t^{(m)} \in \mathcal{W}_{4} \subseteq \Wb$. 
\end{example}

\begin{prop} \label{prop:cornered2} Let $\pi$ be a cornered permutation such that $\pi_{n-2} \pi_{n-1} \pi_{n} = 2n1$.  
Then
$$\B(\pi) = \b(\zeta p^\infty) =  N-1,$$
where $\zeta$ is the unique prefix defined by a $-N$-segmentation with $e_{N-1} = n$.
\end{prop}

\begin{proof}  The proof follows in the same way as Proposition \ref{prop:corneredn}, using now that $p = 0(N{-}2)$.  
 \end{proof}

\begin{example}  Let $\pi = 34251$, a cornered permutation of the form $\pi_{n-2} \pi_{n-1} \pi_{n} = 2 n 1$.  Then $\hat{\pi} = {\star} 5421$, $\asc(\hat{\pi}) = 0$ and $\N(\pi) = 2$.  A valid $-2$-segmentation of $\hat{\pi}$ such that $e_1 = 5$ is given by $(e_0, e_1, e_2) = (0, 5, 5)$, defining the prefix $\zeta = 0000$ with $p= 00$.  By Proposition \ref{prop:cornered2}, $\B(\pi) = 1$.  Moreover, by Lemma \ref{betainduces} and Lemma \ref{lem:inWb}, $t^{(m)} = 0000(00)^{2m} \wmin$, $t^{(m)} \in \Wb$ and $\Pat(t^{(m)}, \Sigma_{-\beta}, n) = \pi$ for all $\beta > 1$ and $m \geq 2$.  \end{example}

 \begin{prop} \label{prop:collapsednlodd} Let $\pi$ be a collapsed permutation such that $n-\ell$ is odd. For $1 \leq i \leq \min\{|p|, |q|\}$, let $\zeta^{(i)}$ be the prefixes obtained by the valid $-N$-segmentations of $\hat{\pi}$.  Define $p^{(i)} = z_{[x, n-1]}^{(i)}$.  Let $\zeta^{(k)}$ be a prefix such that $z_{[\ell, n-1]}^{(k)} p^{(k)}$ is minimal, with respect to $\less$, among these choices of segmentation.
 Then $$\B(\pi) = \b(\zeta^{(k)} (p^{(k)})^\infty).$$  \end{prop}

\begin{proof} To show $\B(\pi) \leq \b(\zeta^{(k)} ({p^{(k)}})^\infty)$, let $\beta > \b(\zeta^{(k)} (p^{(k)})^\infty)$.  By Proposition \ref{prop:pat} there exists an $m_0 \geq \frac{n-1}{2}$ such that for all $m \geq m_0$, we have $s^{(k, m)} \in \Wb$ and $\Pat(s^{(k, m)},\Sigma_{-\beta}, n)=\pi$.

Now we show $\B(\pi) \geq \b(\zeta^{(k)} ({p^{(k)}})^\infty)$.  Let $\b(\zeta^{(k)}({p^{(k)}})^\infty) > \beta > 1$ and let $w \in \Wb$ be a word inducing $\pi$.   By Lemma \ref{beginszeta}, $w$ must start with $\zeta^{(i)}$ for some $1 \leq i \leq \min \{ |p|, |q|\}$.   Write $w = \zeta^{(i)} w_{[n, \infty)}$.  Since $w$ induces $\pi$, we obtain $ w_{[n, \infty)} \less  w_{[x, \infty)} = p^{(i)} w_{[n, \infty)}$, and by Lemma \ref{dinfinity}, we have $ w_{[n, \infty)} \less (p^{(i)})^\infty$.  Since $n - \ell$ is odd, we obtain 
$$w_{[\ell, \infty)} = z^{(i)}_{[\ell, n-1]} w_{[n, \infty)} \gess z^{(i)}_{[\ell, n-1]} (p^{(i)})^\infty \gesseq z^{(k)}_{[\ell, n-1]}(p^{(k)})^\infty.$$
Moreover, Lemma \ref{betab}  implies that $\bar{\beta}(w) \ge \bar{\beta}(z^{(k)}_{[\ell, n-1]} (p^{(k)})^\infty)$ because, by Proposition \ref{prop:bb}, \\ $z^{(k)}_{[\ell, n-1]}(p^{(k)})^\infty$ is the largest subword of $\zeta^{(k)} (p^{(k)})^\infty$. Since Proposition \ref{betcor} tells us $\bar{\beta}(\zeta^{(k)} (p^{(k)})^\infty) = \b(\zeta^{(k)} (p^{(k)})^\infty)$, we may now conclude that $\B(\pi) \ge \b(\zeta^{(k)}(p^{(k)})^\infty).$ 
\end{proof}

\begin{example} Let $\pi = 41853762$, a collapsed permutation such that $n-\ell$ is odd.  Then $\hat{\pi} = 8 {\star} 7 1 3265$ and $\asc(\hat{\pi}) = 2$.  The only minimal segmentation $(e_0, e_1, e_2, e_3) = (0, 4, 6, 8)$ produces $\zeta = 0021021$, which satisfies $q = p^2$, where $p = 021$.  Therefore, $\pi$ is collapsed and $\N(\pi) = 4$.  There are three valid $-4$-segmentations of $\hat{\pi}$ giving rise to the prefixes $\zeta^{(1)} = 1032132$, $\zeta^{(2)} = 0031032$ and $\zeta^{(3)} = 0031021$.   Now we must choose the segmentation such that $z_{[\ell, n-1]}^{(i)}p^{(i)}$ is minimized.  In this case, it is $\zeta^{(1)}$ with this property.   
Since both $h-x$ and $|p|$ are even, we use (\ref{eq:smdef}.1) to give $\wend = z_{[x, h-1]} \wmin = 32 \wmin$. By Proposition \ref{prop:collapsednlodd}, we find $\B(\pi) = \b(\zeta^{(1)}(p^{(1)})^\infty) \approx 3.148$.  Moreover, by Lemma \ref{betainduces}, we obtain:
$s^{(1, m)} = 1032132 (32132)^{2m} 32 \wmin$ and $s^{(1, m)}$ induces $\pi$ for all $\beta > \b(\zeta^{(1)} (p^{(1)})^{ \infty})$ and $m \geq 4$.  Moreover, by Lemma \ref{lem:inWb}, given a $\beta > \b(\zeta^{(1)} (p^{(1)})^\infty)$, there exists an $m_0 \geq 4$ such that, for any $m \geq m_0$, we have $s^{(1, m)} \in \Wb$ and $\Pat(s^{(1, m)}, \Sigma_{-\beta}, n) = \pi$. \end{example}

 \begin{prop} \label{prop:collapsednleven} Let $\pi$ be a collapsed permutation such that $n-\ell$ is even. For $1 \leq i \leq \min\{|p|, |q|\}$, let $\zeta^{(i)}$ be the prefixes obtained by the valid $-N$-segmentations of $\hat{\pi}$.  Define $q^{(i)} = z_{[y, n-1]}^{(i)}$.  Let $\zeta^{(k)}$ be a prefix such that $z_{[\ell, n-1]}^{(k)} q^{(k)}$ is minimal among these choices of segmentation.
Then  $$\B(\pi) = \b(\zeta^{(k)} (q^{(k)})^\infty).$$  \end{prop}

\begin{proof} The proof follows in the same way as Proposition \ref{prop:collapsednlodd}.  \end{proof}

\begin{example} \label{infnotmin} 
Let $\pi = 564132$, a collapsed permutation such that $n-\ell$ is even. Then $\hat{\pi} = 3{\star}2164$ and  $\asc(\hat{\pi}) = 1$.  The only minimal segmentation $(e_0, e_1, e_2) = (0, 4, 6)$ produces $\zeta = 11000$, which satisfies $q = p^2$, where $p = 0$.  Therefore, $\pi$ is collapsed and $\N(\pi) = 3$.  There is only one valid $-3$-segmentation, given by $(e_0, e_1, e_2, e_3) = (0, 1, 4, 6)$.  This defines the prefix $\zeta^{(1)} = 22101$, with $q^{(1)} = 01$.  Since $h-y$ is odd and $|q|$ is even, we use Equation~(\ref{eq:tmdef}.1) to get $\wend =  z_{[y, h-1]} \wmin = 0 \wmin$.  By Proposition \ref{prop:collapsednleven}, $\B(\pi) = \b(\zeta^{(1)}(q^{(1)})^\infty) = 2$. By Lemma \ref{betainduces}, $s^{(1, m)} = 22101(01)^{2m} 0 \wmin$ induces $\pi$ for all $\beta > \b(\zeta^{(1)}(q^{(1)) \infty}) = 2$ and $m \geq 3$.   Moreover, by Lemma \ref{lem:inWb}, for every $\beta > 2$ there exists an $m_0 \geq 3$ such that, for all $m \geq m_0$, we have $s^{(1, m)} \in \Wb$ and $\Pat(s^{(1, m)}, \Sigma_{-\beta}, n) = \pi$.    \end{example}

\begin{thm} Let $\pi \in \S_{n}$, and let $\gamma > \B(\pi)$.  The propositions in this section, summarized in Table \ref{tab:cases},  give a construction for a word $v \in \mathcal{W}_{-\gamma}$ inducing $\pi$.  \end{thm}

\section{Computation of $\B(\pi)$}\label{sec:B}

In this section we find the negative shift-complexity, $\B(\pi)$, of a given permutation $\pi$ by expressing it as the largest real root of a certain polynomial $\P_{\pi}(x)$, in analogy to the construction in \cite{Elibeta} for $\beta$-shifts.  If $w$ is periodic, write $w = (w_{[1, r]})^\infty$ where $r$ is minimal with this property, and define
$$p_{w}(x) =  (-x)^r -1 + \sum_{j = 1}^r (w_{j}+1)(-x)^{r-j}.$$
More generally, if $w$ is eventually periodic, write $w = w_{[1, k]} (w_{[k+1, r]})^\infty$, where $k$ and $r$ are minimal  with this property, and define
$$p_{w}(x) = \left((-x)^{r-k} - 1\right) \left((-x)^{k} + \sum_{i = 1}^{k}(w_i + 1)(-x)^{k-i}\right) + \sum_{j = 1}^{r-k} (w_{k+j} + 1)(-x)^{r-k-j}.$$

Recall from Lemma \ref{numberseg} that, if $\pi$ is regular, there is a unique prefix $\zeta$ arising from a valid $-\N(\pi)$-segmentation of $\hat{\pi}$ as in Definition~\ref{segmentation}, and if $\pi$ is collapsed, there are $\min \{ |p|, |q| \}$ distinct such prefixes.

\begin{thm}\label{thm:P} For any $\pi \in \S_n$ with $n \geq 2$, let $N = \N(\pi)$, and define indices $\ell$, $x$ and $y$ by $\pi_{\ell} = n$, $\pi_x = \pi_n + 1$ (defined only if $\pi_n\neq n$) and $\pi_y = \pi_n - 1$ (defined only if $\pi_n\neq 1$).   Define a polynomial $\P_{\pi}(x)$ in each case as follows.

\begin{itemize}
\item If $\pi$ is regular, let $\zeta$ be the unique prefix arising from a valid $-N$-segmentation of $\hat{\pi}$, and consider three cases:
\begin{itemize} 
\item If $n - \ell$ is odd, let
$$\P_{\pi}(x) = p_{z_{[\ell, n-1]} (z_{[x, n-1]})^\infty}(x). $$
\item If $n-\ell$ is even and $\pi_n = 1$, let
$$\P_{\pi}(x) = p_{(z_{[\ell, n-1]}0)^\infty}(x).$$
\item If $n-\ell$ is even and $\pi_n \neq 1$, let 
$$\P_{\pi}(x) = p_{z_{[\ell, n-1]}(z_{[y, n-1]})^\infty}(x).$$
\end{itemize}
\item If $\pi$ is cornered, let 
$$\P_{\pi}(x) = x - (N-1).$$
\item If $\pi$ is collapsed, consider two cases:
\begin{itemize}
\item If $n-\ell$ is odd, let $\zeta^{(i)}$, for $1 \leq i \leq \min\{|p|, |q|\}$, be the prefixes arising from  valid $-N$-segmentations of $\hat{\pi}$, let $k$ be the value of $i$ that minimizes
$z^{(i)}_{[\ell, n-1]}(z^{(i)}_{[x, n-1]})^\infty$ with respect to $\less$, and let
$$\P_{\pi}(x) = p_{z^{(k)}_{[\ell, n-1]} (z^{(k)}_{[x, n-1]})^\infty}(x).$$
\item If $n-\ell$ is even, let $\zeta^{(i)}$, for $1 \leq i \leq \min \{ |p|, |q| \}$, be the prefixes arising from the valid $-N$-segmentations of $\hat{\pi}$, let $k$ be the value of $i$ that minimizes 
$z^{(i)}_{[\ell, n-1]}(z^{(i)}_{[y, n-1]})^\infty$ with respect to $\less$, and let
$$\P_{\pi}(x) = p_{z^{(k)}_{[\ell, n-1]} (z^{(k)}_{[y, n-1]})^\infty}(x).$$
\end{itemize}
\end{itemize}
Then $\B(\pi)$ is the largest real root $\beta \geq 1$ of $\P_{\pi}(x)$.  
\end{thm}

Notice that $\P_{\pi}(x)$ is always a monic polynomial with integer coefficients.  Moreover, for $\pi \in \S_n$, its degree is never greater than $n-1$.  
\begin{proof}
In the propositions of Section 4, summarized in Table~\ref{tab:cases}, for a given permutation $\pi$ we found a word $w$ such that $\B(\pi) = \b(w)$.   Suppose that we do not have that $n - \ell$ is even and $\pi_n = 1$, a case we will consider separately.   Then Proposition \ref{prop:bb} implies that $\b(w)$ is the largest real solution to $f_{w_{[\ell, \infty)}}(x) = 1$  if $w_{[\ell, \infty)} \gess u$ and $\b(w) = 1$ otherwise.

Now suppose that $n - \ell$ is even and $\pi_n =1$, in which case we found $w = \zeta (0 z_{[\ell, n-1]})^\infty$.  Moreover, by an argument similar to Proposition \ref{prop:bb}, the fact that $\zeta \wmin$ induces $\pi$ whenever $\beta > \b(\zeta (0 z_{[\ell, n-1]})^\infty)$ implies that $w_{[\ell, \infty)} = (z_{[\ell, n-1]}0)^\infty$ is the largest subword of $w$.  Therefore, Definition \ref{barb} implies that $\b(w)$ is the largest real solution to $f_{w_{[\ell, \infty)}}(x) = 1$ whenever $w_{[\ell, \infty)} \gess u$ and $\b(w) = 1$ otherwise.  

By rearranging the expression, finding the largest real solution to $f_{w_{[\ell, \infty)}}(x) = 1$ is equivalent to finding the largest real root of  $p_{w_{[\ell, \infty)}}(x)$.  With the word $w$ determined by each of the propositions in Table~\ref{tab:cases}, these are exactly the polynomials listed above.  Therefore, $\B(\pi)$ is equal to the largest real root $\beta \geq 1$ of $\P_{\pi}(x)$, and equal to $1$ if no such root exists (occurring exactly in the case that $w_{[\ell, \infty)} \less u$).
\end{proof}

\begin{table}[htb]
\begin{tabular}{|c |c|c|c|} 
 \hline
$\pi$ & $\B(\pi)$ & $\P_\pi(\beta)$ \\
 \hline
 1324, 1342, 1432, 2134, 2143, 2314, & 1 & $\beta -1$  \\
 
 2431, 3142, 3214, 3241, 3412, 3421, 4213 & & \\

 \hline
1423, 4231 & 1.618 & $\beta^2 - \beta - 1$ \\
\hline
 2341, 2413, 3124, 4123 & 1.755 & $\beta^3 - 2\beta^2 + \beta - 1$\\
 \hline
 4132 & 1.839 & $\beta^3 - \beta^2 - \beta - 1$ \\
 \hline
1234, 1243, 4312 & 2 & $\beta - 2$ \\
 \hline
4321 & 2.247 & $\beta^3 - 2\beta^2 - \beta + 1$ \\
\hline
\end{tabular}
\caption{The negative shift-complexity of all permutations of length $4$.}
\label{tab:4}
\end{table}

\begin{table}[htb]{\small
\begin{tabular}{|c |c|c|c|} 
 \hline
$\pi$ & $\B(\pi)$ & $\P_\pi(\beta)$ \\
 \hline
$13425, 14352, 14523, 15243, 15324, 15342, 15423, 15432,  21435, 21453,$ & \multirow{3}{*}{$1$} & \multirow{3}{*}{$\beta - 1$} \\
$23145, 23154, 24153, 24315, 24531, 32415, 32541, 34215, 34251, 35214,$ & & \\
$35421,  41523, 41532, 43152, 43512, 43521, 45231, 52143, 52314, 54213$ & & \\

 \hline

$ 52134$ & $1.3247$ & $\beta^3 - \beta - 1$ \\
  \hline 
$14532, 31452, 42135, 45213, 53241$ & $1.4656$ & $\beta^3 - \beta^2 - 1$ \\ 
  \hline 
 
 $14235, 15234, 23514, 25134, 25341, 31524, 32451,$ & \multirow{2}{*}{$1.7549$} &  \multirow{2}{*}{$\beta^3 - 2\beta^2 + \beta - 1$}\\ 

  $34125, 34152, 42315, 42351, 42513, 45321, 51342$ & & \\
  \hline 
$ 15243, 25143, 51324$ &  $1.8393$ & $\beta^3 - \beta^2 - \beta - 1$ \\
  \hline 
 $24351,  32514, 35142, 41325, 52431$ & $1.8832$ & $\beta^4 - 2\beta^3 + \beta^2 - 2\beta + 1$ \\
  \hline 
 $51432$ & $1.89718$ & $\beta^4 - 2\beta^3 + \beta^2 - \beta - 1$ \\
 \hline 
$51423$ & $1.92756$ & $\beta^4 - \beta^3 - \beta^2 - \beta - 1$ \\

\hline
 $12435, 12453, 13245, 13254, 15423,  21345, 21534,  23415, 23541,$  & \multirow{2}{*}{2} & \multirow{2}{*}{$\beta - 2$}\\ $25314, 25413, 31245, 31254, 31542, 32145, 32154, 43251, 45312, 54132$ &  &  \\ 
  \hline 
$54123$ & $2.1479$ & $\beta^3 - \beta^2 - 2\beta - 1$ \\ 
 \hline 
$53124$ & $2.17872$  & $\beta^4 - 3\beta^3 + 2\beta^2 - 1$ \\
 \hline 
$13542, 35412, 41253, 43125, 54231$ & $2.2056$  & $\beta^3 - 2\beta^2 - 1$ \\
  \hline 
$53241$ & $2.2938$ & $\beta^4 - 3\beta^3 + 2\beta^2 - 2$ \\
 \hline 
$12543, 14253, 14325, 21543, 25431, 42153, 43215$  & $2.3247$ & $\beta^3 - 3\beta^2 + 2\beta - 1$ \\
 \hline 

$12534$ & $2.4142$ & $\beta^2 - 2\beta - 1$\\
\hline
$53421$ & $2.44868$ & $\beta^4 - 3\beta^3 + 2\beta^2 - 2\beta + 1$ \\
 \hline 
$53412$ & $2.47098$ & $\beta^4 - 2\beta^3 - \beta^2 - 1$ \\
 \hline 
$52413$ & $2.48753$ & $\beta^4 - 3\beta^3 + 2\beta^2 - \beta - 2$ \\
 \hline 
 $23451, 24513, 45123, 41235, 52341$ & $2.5214$ & $\beta^4 - 3\beta^3 + \beta^2 - 2\beta + 2$ \\
  \hline 
 $13452, 13524$ & $2.5214$ & $\beta^3 - 3\beta^2 + 2\beta - 2$ \\
  \hline 
 $41352, 42531, 31425, 34521$ & $2.6180$ & $\beta^2 - 3\beta + 1$ \\
  \hline 
$53142$ & $2.66577$ & $\beta^4 - 3\beta^3 + 2\beta + 1$ \\
 \hline 
 $34512$ & $2.7321$ & $\beta^2 - 2\beta - 2$ \\
  \hline 
$35124$ & $2.7693$ & $\beta^3 - 3\beta^2 + \beta - 1$ \\
\hline
$51243$ & $2.7769$ & $\beta^4 - 3\beta^3 + \beta + 2$ \\
 \hline 
$51234$ & $2.79714$ & $\beta^4 - 2\beta^3 - 2\beta^2 - \beta + 1$ \\
 \hline 
 $45132$ & $2.8312$ & $\beta^3 - 2\beta^2 - 2\beta -1$ \\
 \hline
 $35241$ & $2.8794$ & $\beta^3 - 3\beta^2 + 1$ \\
 \hline 
$24135$ & $2.8933$ & $\beta^3 - 3\beta^2 + \beta - 2$ \\
 \hline 
$12345, 12354$ & $3$ &$ \beta - 3$ \\
 \hline 
$54321$ & $3.23402$ & $\beta^4 - 4\beta^3 + 3\beta^2 - 2\beta + 1$ \\
\hline
$54312$ & $3.24262$ & $\beta^4 - 3\beta^3 - \beta^2 + \beta - 1$ \\
\hline
\end{tabular}}
\caption{The negative shift-complexity of all permutations of length $5$.}
\label{tab:5}
\end{table}

\begin{example}  For $\pi = 15237864$ as in Example \ref{nleven}, we had $t^{(m)} = 0101332(1332)^{2m} 1 \wmax$.  Therefore, $z_{[\ell, n-1]} q^\infty = (3213)^\infty$.  We obtain $$p_{(3213)^\infty}(x) = (-x)^4 -1 + 4(-x)^3 + 3(-x)^2 + 2(-x) + 4.$$ It follows from Theorem~\ref{thm:P} that $\B(\pi)$ is the largest real root of
$$P_{\pi}(\beta) = \beta^4  - 4\beta^3 + 3\beta^2 - 2\beta + 3,$$
and $\B(\pi) \approx 3.154$. 
\end{example}

For permutations in $\S_3$, we have $B(\pi) = 1$ if $\pi \in\{123, 132, 213, 231, 321\}$.   The remaining permutation $\pi = 312$ has $\B(\pi) \approx 1.618$, the largest real root of the polynomial $\P_{\pi}(\beta) = \beta^2 - \beta - 1$.  

Carrying out the computations for all $\pi\in\S_4$ and $\pi \in \S_5$ we obtain Tables~\ref{tab:4} and~\ref{tab:5}.

\section*{Acknowledgments}
The authors thank Wolfgang Steiner for useful comments and suggestions.
The first author was partially supported by grant \#280575 from the Simons Foundation and by grant H98230-14-1-0125 from the NSA.

\end{document}